\documentclass{article}

\usepackage{a4wide}
\usepackage{amsmath}
\usepackage{amsthm}
\usepackage{amssymb}
\usepackage[colorlinks=true,urlcolor=blue,linkcolor=blue,plainpages=false,pdfpagelabels
]{hyperref}

\newtheorem{theorem}{Theorem}[section]

\newtheorem{proposition}[theorem]{Proposition}
\newtheorem{corollary}[theorem]{Corollary}
\newtheorem{lemma}[theorem]{Lemma}
\newtheorem{fact}[theorem]{Remark}
\newtheorem{remark}[theorem]{Remark}

\newcommand{\N}{{\mathbb N}}
\newcommand{\R}{{\mathbb R}}
\newcommand{\remin}{\mathop{-\!\!\!\!\!\hspace*{1mm}\raisebox{0.5mm}{$\cdot$}}\nolimits}

\title{Quantitative results on a generalized viscosity approximation method}
\author{Paulo Firmino${}^{a}$ and Lauren{\c t}iu Leu{\c s}tean${}^{b,c,d}$\\[2mm]
\footnotesize ${}^{a}$ Departamento de Matem\'atica, Faculdade de Ci\^encias, Universidade de Lisboa\\ 
\footnotesize ${}^{b}$ LOS, Faculty of Mathematics and Computer Science, University of Bucharest\\
\footnotesize ${}^{c}$ Simion Stoilow Institute of Mathematics of the Romanian Academy\\
\footnotesize ${}^{d}$ Institute for Logic and Data Science, Bucharest\\[1mm]
\footnotesize Emails: \protect\url{fc49883@alunos.ciencias.ulisboa.pt }, \protect\url{laurentiu.leustean@unibuc.ro}
}

\begin{document}
\date{}
\maketitle

\begin{abstract}

\noindent
In this paper,  we study, in a nonlinear setting, the asymptotic behaviour of a generalized viscosity approximation method associated with a countable family of nonexpansive mappings satisfying resolvent-like conditions.  We apply proof mining methods to obtain quantitative results on asymptotic regularity  in $W$-hyperbolic spaces and rates of metastability in CAT(0) spaces. \\



\end{abstract}

\section{Introduction}

Let $X$ be a $W$-space, $C$ a nonempty convex subset of $X$ and  $(T_n:C\to C)_{n\in\N}$  a sequence of nonexpansive mappings with common fixed points.
The generalized Viscosity Approximation Method (genVAM, for short) is defined as follows: 
\[
x_0=x\in C, \quad x_{n+1}=\alpha_n f(x_n) +(1-\alpha_n)T_nx_n,
\]
where $(\alpha_n)_{n\in\N}$ is a sequence in $[0,1]$ and $f:C\to C$ is an $\alpha$-contraction 
with $\alpha\in[0,1)$. 

If $f$ is constant, the genVAM becomes the  abstract Halpern-type Proximal Point Algorithm (abstract HPPA, for short), a 
generalization of the well-known Halpern-type Proximal Point Algorithm (HPPA), which was introduced by 
Xu \cite{Xu02} and by Kamimura and Takahashi \cite{KamTak00}. The abstract HPPA has been studied recently by Aoyama, Kimura, Takahashi and 
Toyoda \cite{AoyKimTakToy07} and, using proof mining methods, by Sipo\c s \cite{Sip22} and by Kohlenbach and Pinto \cite{KohPin22}.  
The Viscosity Approximation Method (VAM) in Banach spaces,  due to Xu et al. \cite{XuAltAlzChe22}, is also a particular case of the genVAM iteration.

In this paper we study the asymptotic behaviour of the genVAM iteration by assuming that the family 
$(T_n)$ satisfies resolvent-like conditions introduced by the second author, Nicolae and Sipo\c s  \cite{LeuNicSip18}
in their abstract analysis of the proximal point algorithm. Resolvents of maximally monotone operators in Hilbert spaces, as well as resolvents  
of $m$-accretive operators in Banach spaces, are examples of families of mappings that satisfy these conditions. 

In Section \ref{section-rates-as-reg} we study the genVAM  iteration in $W$-hyperbolic spaces  and compute rates of ($(T_n)$-)asymptotic regularity, as well as rates of 
$T_m$-asymptotic regularity for all $m\in\N$.  In this way, we  extend to our very general setting quantitative asymptotic results obtained for the VAM in Banach 
spaces by the authors  \cite{FirLeu25} and for the HPPA in Hilbert spaces by 
the second author and Pinto \cite{LeuPin21}. Furthermore, as an application of a lemma due to Sabach and Shtern \cite{SabSht17}, linear such rates are computed 
for particular choices of the parameter sequences.

In Section \ref{rates-meta-thms}, we first obtain, in the setting of complete CAT(0) spaces, uniform rates of metastability for the abstract HPPA, 
thereby generalizing - while providing simpler proofs and simpler rates - recent results of Sipo\c{s} \cite{Sip22}. 
Furthermore, by applying a method developed by Kohlenbach and Pinto \cite{KohPin22}, we derive rates of metastability 
for the genVAM iteration and, as an immediate consequence, establish a (qualitative) strong convergence result for this iteration.

The results presented in this paper are obtained using techniques from proof mining, a research field concerned with 
extracting new quantitative and qualitative information from mathematical proofs through proof-theoretic methods. 
A detailed presentation of the proof mining program can be found in Kohlenbach's monograph \cite{Koh08}, and more recent 
applications of proof mining are surveyed in \cite{Koh19,Koh20}.

We recall below some notation and notions used throughout the paper. We denote $\N^*=\N\setminus \{0\}$ and $\R^+=(0,\infty)$.
For all $m,n\in\N$, we set $m\remin n = \max\{m-n,0\}$. Moreover, if $n\geq m$, we write $[m;n]=\{m,m+1,\ldots, n\}$. 
For $x\in \R$, $\lceil x \rceil$ denotes the ceiling of $x$. For any mapping $T$, $Fix(T)$ is the set of fixed points of $T$.
For every $g:\N\to \N$, the mappings $\tilde{g}, g^+:\N\to \N$ are defined as follows:
\[\tilde{g}(n)=n+g(n), \quad g^+(n)=\max\{g(i)\mid i\in [0;n]\}.\]
Furthermore, if $n\in\N$, $g^{(n)}$ denotes the $n$-th iterate of $g$.

\section{Preliminaries}

\subsection{$W$-hyperbolic and CAT(0) spaces}

A $W$-space \cite{CheLeu22} is a structure $(X,d,W)$, where $(X,d)$ is a metric space and 
$W:X\times X\times [0,1]\to X$. 
These structures have been already considered by Takahashi \cite{Tak70}, $W$ being called a 
\emph{convexity mapping}, as $W(x,y,\lambda)$ is thought of as an abstract convex 
combination of the points $x,y\in X$ with parameter $\lambda$. 
We denote a $W$-space simply by $X$. A nonempty subset $C$ of $X$ is 
said to be convex if for every $x,y\in C$, $ W(x,y,\lambda) \in C$ for all  
$\lambda \in [0,1]$.

A $W$-hyperbolic space \cite{Koh05} is a $W$-space $X$ satisfying for all $x,y,z,w\in X$ and all $\lambda,\tilde{\lambda}\in [0,1]$,
\begin{eqnarray*}
(W1) & d(z,W(x,y,\lambda))\le (1-\lambda)d(z,x)+\lambda d(z,y),\\
(W2) & d(W(x,y,\lambda),W(x,y,\tilde{\lambda}))=|\lambda-\tilde{\lambda}|d(x,y),\\
(W3) & W(x,y,\lambda)=W(y,x,1-\lambda),\\
(W4) & \,\,\,d(W(x,z,\lambda),W(y,w,\lambda)) \le (1-\lambda)d(x,y)+\lambda
d(z,w).
\end {eqnarray*}
A normed space is, obviously, a $W$-hyperbolic space: just define $W(x,y,\lambda)=(1-\lambda)x+\lambda y$. Important classes of $W$-hyperbolic spaces are 
Busemann spaces \cite{Bus48,Pap05}, UCW-hyperbolic spaces \cite{Leu07,Leu10} and CAT(0) spaces \cite{BriHae99,AleKapPet19}. As proved in \cite[pp. 386-388]{Koh08}, a CAT(0) space can be defined as a $W$-hyperbolic space $X$ satisfying 
for all $x,y,z\in X$, 
\begin{equation*}
d^2\left(z,W\left(x,y,\frac{1}{2}\right)\right) \le \frac{1}{2}d^2(z,x) + \frac{1}{2} d^2(z,y)-\frac{1}{4}d^2(x,y).
\end{equation*}

We shall use from now on the notation $(1-\lambda)x+\lambda y$ for $W(x,y,\lambda)$.

\begin{lemma}
Let $X$ be a $W$-hyperbolic space. For all $x,y,w,z\in X$, $\lambda, \widetilde{\lambda}\in[0,1]$, 
\begin{align}
d(x,(1-\lambda)x + \lambda y)=\lambda d(x,y) \text{ and } 
d(y,(1-\lambda)x + \lambda y)=(1-\lambda)d(x,y), \label{W10-W-hyp} \\
\lambda x + (1-\lambda) x=x, \quad 1x+0y=x, \quad 0x+1y=y, \label{W-hyp-trivial}\\
d((1 - \lambda) x + \lambda z, (1 - \widetilde{\lambda}) y + \widetilde{\lambda} w)  \leq (1 - \lambda) d(x, y) + \lambda d(z, w) + 
\left\vert\lambda - \widetilde{\lambda} \right\vert d(y, w), \label{W42-xyzw-ineq1-1} \\
d(\lambda x + (1 - \lambda)z, \widetilde{\lambda} y +  (1 - \widetilde{\lambda}) w)  \leq \lambda d(x, y) +  (1 - \lambda)d(z, w) + 
\left\vert\lambda - \widetilde{\lambda}\right\vert d(y, w). \label{W42-xyzw-ineq1-2} 
\end{align}
\end{lemma}
\begin{proof} \eqref{W10-W-hyp} and \eqref{W-hyp-trivial} were proved by 
Takahashi \cite{Tak70} for $W$-spaces satisfying (W1). See also \cite[Proposition 17.17]{Koh08}. 
\eqref{W42-xyzw-ineq1-1}  is proved in \cite[Lemma~2.1(iv)]{CheLeu22}. Apply \eqref{W42-xyzw-ineq1-1}  with 
$\lambda:=1-\lambda$ and $\tilde{\lambda}:=1-\tilde{\lambda}$
to get \eqref{W42-xyzw-ineq1-2}.
\end{proof}

\subsection{Quantitative notions and lemmas}

Let $(x_n)_{n\in\N}$ be a sequence in a metric space $(X,d)$ and $\varphi:\N\to\N$. If $(x_n)$ is Cauchy, then 
$\varphi$ is a Cauchy modulus of $(x_n)$ if for all $k,n\in\N$  with $n\geq \varphi(k)$,
\[\forall p\in\N \left(d(x_{n+p},x_n)\leq \frac1{k+1}\right).\]
If $\lim\limits_{n\to \infty} x_n = x \in X$, then $\varphi$ is a rate of convergence of 
$(x_n)$ if for all $k,n\in\N$ with $n\geq \varphi(k)$,
\[d(x_n,x)\leq \frac1{k+1}.\]

\begin{lemma}\label{rate-conv-implies-Cauchy-modulus}
Assume that $\lim\limits_{n\to \infty} x_n = x$ with rate of convergence $\varphi$. Define
\[\varphi^*:\N\to\N, \quad \varphi^*(k)=\varphi(2k+1).\]
Then $\varphi^*$ is a Cauchy modulus of $(x_n)$. 
\end{lemma}
\begin{proof}
For every $k,n,p \in \N$ with $n\geq \varphi^*(k)$, we have that 
\[ d(x_{n+p},x_n) \leq d(x_{n+p},x) + d(x_n,x) \leq \frac1{2(k+1)} + \frac1{2(k+1)} = \frac1{k+1}.\]
\end{proof}

Let $\sum\limits_{n=0}^\infty a_n$ be a series of nonnegative real numbers and 
$\left(\tilde{a}_n\right)$ be its sequence of partial sums, 
that is $\tilde{a}_n=\sum\limits_{i=0}^{n} a_i$ for all $n\in\N$.  If $\sum\limits_{n=0}^\infty a_n$ 
converges, then by a Cauchy modulus of $\sum\limits_{n=0}^\infty a_n$ we understand a Cauchy modulus of $\left(\tilde{a}_n\right)$. 
If $\sum\limits_{n=0}^\infty a_n$ diverges, then a rate of divergence of $\sum\limits_{n=0}^\infty a_n$ is a mapping $\theta:\N\to\N$  satisfying 
$\sum\limits_{i=0}^{\theta(n)} a_i \geq n$ for all $n\in\N$.

\begin{lemma}\label{rate-divergence-an-prop}
Assume that $\sum\limits_{n=0}^\infty a_n$ diverges with rate of divergence $\theta$.
\begin{enumerate}
\item\label{rate-divergence-an-01}  If $(a_n) \subseteq [0,1]$, then $\theta(n)\geq n-1$ for all $n\in\N$.
\item\label{rate-divergence-an-N}  For all $N\in\N^*$, $\sum\limits_{n=0}^\infty a_{n+N}$ diverges 
with rate of divergence 
$\theta^*(n)=\theta\left(n+\left\lceil \sum_{i=0}^{N-1} a_i\right\rceil\right)\remin N$. If
$(a_n) \subseteq [0,1]$, then we can take $\theta^*(n)=\theta^+(n+N)\remin N$.
\item\label{rate-divergence-cbn-np1-cbn}  For all $c\in (0,\infty)$, $\sum\limits_{n=0}^\infty ca_n$ diverges with rate of 
divergence $\chi(n)=\theta\left(\left\lceil \frac{n}{c}\right\rceil\right)$.
\item\label{rate-divergence-sum-bn} If $\sum\limits_{n=0}^\infty b_n$ diverges with rate of divergence $\gamma$, then 
$\sum\limits_{n=0}^\infty (a_n+b_n)$ diverges with rate of divergence 
$\omega(n)=\min\{\theta(n),\gamma(n)\}$.
\end{enumerate}
\end{lemma}
\begin{proof}
\begin{enumerate}
\item Let $n\in\N$ and assume, by contradiction, that $\theta(n) \leq n-2$. It follows that 
$\sum\limits_{i=0}^{\theta(n)} a_i  \leq \sum\limits_{i=0}^{n-2} a_i \leq  n-1$, which is a contradiction.
 
\item Let us denote, for simplicity, $A=\sum\limits_{i=0}^{N-1} a_i $. For all $n\in \N$, 
\begin{align*}
\sum\limits_{i=0}^{\theta^*(n)} a_{i+N} &= \sum\limits_{i=N}^{\theta^*(n)+N} a_i \\
&\geq \sum\limits_{i=N}^{\theta(n+\left\lceil A \right\rceil)} a_i \quad \text{as } 
\theta^*(n)+N=(\theta(n+\left\lceil A \right\rceil)\remin N) + N \geq \theta(n+\left\lceil A \right\rceil) \\
&= \sum\limits_{i=0}^{\theta(n+\left\lceil A \right\rceil)}a_i - A
\geq n + \left\lceil A \right\rceil - \left\lceil A \right\rceil \geq n.
\end{align*}
If $(a_n) \subseteq [0,1]$, then  $\left\lceil A \right\rceil\leq N$, hence,  
$\theta(n+\left\lceil A \right\rceil)\leq \theta^+(n+\left\lceil A \right\rceil)\leq \theta^+(n+N)$, as $ \theta^+$ is increasing. It follows that 
$\sum\limits_{i=0}^{\theta^+(n+N)\remin N} a_{i+N} 
\geq \sum\limits_{i=0}^{\theta(n+\left\lceil A \right\rceil)\remin N} a_{i+N} 
\geq n$. 
	
\item For all $n\in \N$, 
$\sum\limits_{i=0}^{\chi(n)} ca_i = c\sum\limits_{i=0}^{\chi(n)} a_i \geq 
c\left\lceil\frac{n	}{c}\right\rceil\geq n$.

\item For all $n\in\N$,
\[\sum\limits_{i=0}^{\gamma(n)} (a_i+b_i)\ge \sum\limits_{i=0}^{\gamma(n)} a_i \ge n \quad \text{and} \quad
\sum\limits_{i=0}^{\theta(n)} (a_i+b_i)\ge \sum\limits_{i=0}^{\theta(n)} b_i \ge n,\]
\end{enumerate}
\end{proof}

\begin{lemma}\label{cauchy-conv-rate-linear-comb}
Let $(a_n)_{n\in\N}$, $(b_n)_{n\in\N}$ be sequences of real numbers, $q,r\in (0,\infty)$ and 
$\varphi_1, \varphi_2:\N\to \N$. Define $c_n=qa_n+rb_n$ for all $n\in \N$,	 and
\[\varphi:\N\to\N, \quad \varphi(k)=\max\left\{\varphi_1(\lceil 2q(k+1)\rceil-1), \varphi_2(\lceil 2r(k+1)\rceil-1)\right\}.\]
\begin{enumerate}
\item\label{cauchy-conv-rate-linear-comb-cauchy} If $\varphi_1$ is a Cauchy modulus of $(a_n)$ and 
$\varphi_2$ is a Cauchy modulus of $(b_n)$, then $\varphi$ is a Cauchy modulus of $(c_n)$. 
\item\label{cauchy-conv-rate-linear-comb-conv} If  $\lim\limits_{n \to \infty} a_n=a\in\R$ with 
rate of convergence  $\varphi_1$ and  $\lim\limits_{n \to \infty} b_n=b\in\R$ with rate of convergence $\varphi_2$, then 
$\lim\limits_{n \to \infty} c_n=qa+rb$ with rate of convergence $\varphi$.
\item\label{series-Cauchy-linear-comb} If $(a_n), (b_n) \subseteq [0,\infty)$, $\varphi_1$ is a Cauchy modulus of $\sum\limits_{n=0}^\infty a_n$ and 
$\varphi_2$ is a Cauchy modulus of $\sum\limits_{n=0}^\infty b_n$, then $\varphi$ is a Cauchy modulus of $\sum\limits_{n=0}^\infty c_n$.
\end{enumerate}
\end{lemma}
\begin{proof}
\begin{enumerate}
\item Adapt the proof of \cite[Lemma 3.2]{FirLeu25a}.
\item Let $k,n\in\N$ with $n\geq \varphi(k)$. We get that
\begin{align*}
|c_n-(qa+rb)| &=  |q(a_n-a) + r(b_n-b)|  \leq  q|a_n-a| + r|b_n-b| \\
& \leq   q\cdot\frac1{\lceil 2q(k+1)\rceil} + r\cdot\frac1{\lceil 2r(k+1)\rceil} \leq \frac1{k+1}.
\end{align*}
\item Apply \eqref{cauchy-conv-rate-linear-comb-cauchy}.
\end{enumerate}
\end{proof}

The following lemma gives quantitative versions of the well-known Xu's lemma \cite{Xu02}.

\begin{lemma}\label{quant-lem-Xu02-bncn}
Let $(a_n)_{n\in\N}, (b_n)_{n\in\N}, (c_n)_{n\in\N}, (s_n)_{n\in\N}$ be sequences of real numbers
such that $(a_n) \subseteq [0,1]$ and $(c_n), (s_n) \subseteq [0,\infty)$. Assume that $L\in\N^*$ is an upper bound on $(s_n)$ 
and $\sum\limits_{n=0}^\infty a_n$  diverges with rate of divergence $\theta$.
\begin{enumerate}
\item\label{quant-lem-Xu02-bncn-cn-0} Suppose that $s_{n+1}\leq (1-a_n)s_n + a_nb_n$ for all $n\in\N$ and 
that $\psi:\N\to\N$ satisfies the following: 
for all $k,n\in \N$  with $n \ge \psi(k)$, $b_n \le \frac{1}{k+1}$.  

Then $\lim\limits_{n \to \infty} s_n=0$ with rate of convergence $\Sigma$ defined by 
\begin{equation*}
\Sigma(k)=\theta\big(\psi(2k+1)+\lceil \ln(2L(k+1))\rceil\big)+1.
\end{equation*}
\item\label{quant-lem-Xu02-bncn-bn-0} Assume that $s_{n+1}\leq (1-a_n)s_n + c_n$ for all $n\in\N$ 
and  that $\sum\limits_{n=0}^\infty c_n$ converges with Cauchy modulus $\chi$.

Then $\lim\limits_{n \to \infty} s_n=0$ with rate of convergence $\Sigma$ defined by 
\begin{equation*}
\Sigma(k)=\theta\big(\chi(2k+1)+1+\lceil \ln(2L(k+1))\rceil\big)+1.
\end{equation*}
\end{enumerate}
\end{lemma}
\begin{proof}
For (i) see \cite[Proposition 4]{LeuPin21}. 
For (ii) see \cite[Proposition 2.7]{LeuPin24}, a reformulation of \cite[Lemma 2.9(1)]{DinPin23}.
\end{proof}

We shall also use a slight reformulation of a  lemma due to Sabach and Shtern \cite[Lemma~3]{SabSht17}. 
 
\begin{lemma}\label{lem:sabach-shtern-v2}\cite[Lemma~2.8]{LeuPin24}\\
Let $L > 0$, $J \geq N \geq 2$, $\gamma \in (0, 1]$. Assume that $a_n = \frac{N}{\gamma(n + J)}$ for all $n\in\N$,  
$(c_n)_{n\in\N}$ is a sequence bounded above by $L$, $(s_n)_{n\in\N}$ is a sequence of nonnegative 
real numbers with $s_0 \leq L$ and that for all $n \in \N$,
\begin{align}
s_{n + 1} \leq (1 - \gamma a_{n + 1}) s_n + (a_n - a_{n + 1}) c_n. \label{eq:sabach-shtern-main-ineq}
\end{align}
Then, for all $n \in \N$,
\[s_n \leq \frac{J L}{\gamma(n + J)}.\]
\end{lemma}

\subsubsection{Rates of metastability}

Let $(X,d)$ be a metric space and $(x_n)_{n\in \N}$ be a sequence in $X$. We say that $(x_n)$ is metastable \cite{Tao08a} if
\begin{align}\label{def-metastability}
\forall k\in\N\, \forall g:\N \to \N\, \exists N\in\N\, \forall i,j \in [N;N + g(N)] \left(d(x_i, x_j) \leq \frac{1}{k + 1}\right).
\end{align}
One can easily see that $(x_n)$ is Cauchy if and only if $(x_n)$ is metastable. As pointed out by Tao \cite{Tao08a},
metastability is a concept in ``hard analysis" corresponding to Cauchyness, a concept in ``soft analysis". 
In mathematical logic, metastability of $(x_n)$ is known as the Herbrand normal form or the no-counterexample 
interpretation of Cauchyness. We refer to \cite{Koh08} for logical discussions on metastability. 

Rates of metastability are used to express quantitative versions of metastability. A rate of metastability of $(x_n)$ is a 
mapping $\Omega:\N\times \N^\N\to \N$  satisfying the following:
\begin{align}\label{def-rate-metastability}
\forall k\in\N\, \forall g:\N \to \N\, \exists N\leq \Omega(k,g)\, \forall i,j \in [N;N + g(N)] \left(d(x_i, x_j) \leq \frac{1}{k + 1}\right).
\end{align}
Obviously, if $(x_n)$ is Cauchy with Cauchy modulus $\varphi$, then the mapping $\Omega:\N\times \N^\N\to \N$, $\Omega(k,g)=\varphi(k)$ 
is a rate of metastability of $(x_n)$. The converse does not hold, from a rate of metastability of $(x_n)$ 
one cannot, in general, compute a Cauchy modulus of $(x_n)$.

Following \cite{LeuPin21}, for any $L\in\N^*$,  a rate of $L$-metastability of $(x_n)$
is a mapping $\Phi_L:\N\to \N$ satisfying \eqref{def-rate-metastability} with $g(n)=L\in\N$. Thus, $\Phi_L$ satisfies 
\begin{align}\label{def-rate-L-metastability}
\forall k\in \N \, \,\exists N\leq \Phi_L(k)\, \,\forall i, j\in [N;N+L]\, 
\left(d(x_i, x_j)\leq \frac{1}{k+1}\right).
\end{align}

\begin{fact}\label{relation-rates-meta-KohPin22}
In  this paper we use methods developed by Kohlenbach and Pinto \cite{KohPin22} to 
compute rates of metastability. In \cite{KohPin22}, a rate of metastability of $(x_n)$ is a mapping 
$\Omega:\R^+\times \N^\N\to \N$ satisfying 
\begin{align}\label{def-metastability-KohPin22}
\forall \varepsilon >0,g:\N \to \N\, \exists N\leq \Omega(\varepsilon,g) \,\forall i,j \in [N;g(N)] \left(d(x_i, x_j) \leq 
\varepsilon\right).
\end{align}
If $\Omega$ is a rate of metastability in our sense, then the mapping 
\[\Omega^*: \R^+\times \N^\N\to \N, \quad \Omega^*(\varepsilon,g)=\Omega\left(\left\lceil \frac1\varepsilon\right\rceil -1,g\right)\]
 satisfies \eqref{def-metastability-KohPin22}. 
Conversely, let $\Omega$ satisfy \eqref{def-metastability-KohPin22} and define
$\tilde{\Omega}:\N\times \N^\N\to \N$ as follows:
\begin{align*}
\tilde{\Omega}(k,g)=\Omega\left(\frac1{k+1},\tilde{g}\right), \qquad \text{where } \tilde{g}(n)=n+g(n).
\end{align*}
Then $\tilde{\Omega}$ is a rate of metastability in our sense.
\end{fact}

The following lemma adapts to our setting results from \cite{KohPin22}.

\begin{lemma}\label{rate-meta-H-Tn}
Let $\Omega$ be a rate of metastability of $(x_n)$ and define 
\begin{align*}
\Omega^+:\N\times \N^\N\times \N \to \N, & \quad  \Omega^+(k,g,n) = n  + \Omega(k,g_n), \\
\Omega^\dag:\N\times \N^\N\times \N \to \N, & \quad  
\Omega^\dag(k,g,n)=\max\{\Omega^+(k,g,i) \mid i\in[0;n]\},
\end{align*}
where $g_n:\N\to\N, \, \, g_n(m)=n+g(n+m)$.

The following hold:
\begin{enumerate}
\item For all  $k,n\in\N$, $g:\N \to \N$, 
\begin{align}\label{def-metastability-plus}
\exists N\in [n;\Omega^+(k,g,n)]\,\forall i,j \in [N;N+g(N)] \left(d(x_i, x_j) \leq \frac{1}{k + 1}\right).
\end{align}
\item\label{def-metastability-dag} holds with $\Omega^\dag$ instead of $\Omega^+$ and 
$\Omega^\dag$ is monotone in the third argument, that is: for all $k\in \N$, $g:\N\to \N$ and all 
$n,m\in \N$ with $n\leq m$, $\Omega^\dag(k,g,n)\leq \Omega^\dag(k,g,m)$.
\end{enumerate}
\end{lemma}
\begin{proof} 
\begin{enumerate}
\item Let $k,n\in\N$ and $g:\N\to \N$ be arbitrary. As $\Omega$ is a rate of metastability  
of $(x_n)$, there exists  $N'\leq \Omega(k,g_n)$ such that 
\begin{align*}
\forall i,j \in [N';N'+ g_n(N')] \left(d(x_i, x_j) \leq \frac{1}{k + 1}\right).
\end{align*}
Define $N=n+N'$. Then $N\in [n;n+\Omega(k,g_n)]=[n;\Omega^+(k,g,n)]$. Since $g_n(N')=n+g(n+N')=n+g(N)$,
we get that $N'\leq N \leq N+g(N) = n+N'+g(N) = N'+g_n(N')$, hence $[N;N+g(N)] \subseteq [N';N'+g_n(N')]$.
Thus, \eqref{def-metastability-plus} holds.
\item Obviously.
\end{enumerate}
\end{proof}

The following lemma, proved in \cite{LeuPin21}, will be very useful in Section \ref{rates-meta-thms}.

\begin{proposition} \label{xnyn-lemma-meta}
Let $(x_n)_{n\in \N}$, $(y_n)_{n\in \N}$ be sequences in $X$ such that 
$(x_n)$ is Cauchy with rate of metastability $\Omega$ and 
$\lim\limits_{n\to\infty} d(x_n,y_n)=0$ with rate of convergence $\varphi$.

Then $(y_n)$ is Cauchy with rate of metastability $\Gamma$ given by
\begin{equation*}
\Gamma(k,g)=\max\{\varphi(3k+2),\Omega(3k+2,h_{k,g})\},
\end{equation*}
where, for $k\in \N$ and $g:\N \to \N$, 
\begin{equation*}
h_{k,g}:\N\to \N, \quad h_{k,g}(m)= \max\{\varphi(3k+2),m\}-m+g(\max\{\varphi(3k+2),m\}).
\end{equation*}
\end{proposition}
\begin{proof}
Extend the proof of \cite[Proposition 7]{LeuPin21} to metric spaces.
\end{proof}

\subsection{Quantitative asymptotic regularity}

Asymptotic regularity is a well-known notion in optimization, nonlinear analysis and metric 
fixed point theory, introduced by Browder and Petryshyn in the 1960s \cite{BroPet66} for 
the Picard iteration of nonexpansive mappings in metric spaces. In the 1990s, Borwein, Reich and Shafrir 
\cite{BorReiSha92} extended the definition of asymptotic regularity to nonlinear iterations.

Let $(X,d)$ be a metric space, $C$ a nonempty subset of $X$, $(x_n)_{n\in\N}$ a sequence in $C$,
$T:C\to C$ a mapping, and $(T_n:C\to C)_{n\in \N}$ a sequence of mappings. We shall use in 
this paper the following asymptotic regularity-type notions, together with their quantitative versions:
\begin{enumerate}
\item $(x_n)$ is asymptotically regular if $\lim\limits_{n\to \infty} d(x_n,x_{n+1})=0$. A rate of asymptotic regularity of $(x_n)$ is a rate 
of convergence to $0$ of $\big(d(x_n,x_{n+1})\big)$.
\item $(x_n)$ is $T$-asymptotically regular if $\lim\limits_{n\to \infty} d(x_n,Tx_n)=0$. A rate of $T$-asymptotic regularity of $(x_n)$ is a rate 
of convergence to $0$ of $\big(d(x_n,Tx_n)\big)$.
\item $(x_n)$ is $(T_n)$-asymptotically regular if $\lim\limits_{n\to \infty} d(x_n,T_nx_n)=0$. A rate of $(T_n)$-asymptotic regularity of $(x_n)$ is a rate 
of convergence to $0$ of $\big(d(x_n,T_nx_n)\big)$.
\end{enumerate}

\begin{lemma}\label{as-reg-L-meta}
Assume that $(x_n)$ is asymptotically regular with rate $\varphi$. Then for any $L\in \N^*$, 
\[\varphi_L: \N \to \N, \qquad \varphi_L(k)=\varphi(L(k+1)-1)\]
is a rate of $L$-metastability of $(x_n)$.
\end{lemma}
\begin{proof}
 Let $k\in\N$. Obviously,  $d(x_l,x_{l+1}) \leq \frac{1}{L(k+1)}$ for  $l\geq \varphi_L(k)$. 
We get that  for all $i,j\in [\varphi_L(k),\varphi_L(k)+L]$,
\begin{align*}
d(x_i,x_j) & \le \sum_{l=i}^{j-1} d(x_l,x_{l+1}) \le  
\sum_{l=\varphi_L(k)}^{\varphi_L(k)+L-1} d(x_l,x_{l+1}) 
\le  \sum_{l=\varphi_L(k)}^{\varphi_L(k)+L-1} \frac1{L(k+1)}  = 
\frac{1}{k+1}.
\end{align*}
Thus, \eqref{def-rate-L-metastability} holds with $N=\varphi_L(k)$.
\end{proof}

\section{genVAM in $W$-hyperbolic spaces}\label{genVAM-hyperbolic}

Throughout this section, $X$ is a $W$-hyperbolic space,  $\emptyset \ne C \subseteq X$ is a convex subset of $X$,  
$f:C\to C$ is an $\alpha$-contraction with $\alpha\in[0,1)$ and $(T_n:C\to C)_{n\in \N}$ is a sequence of nonexpansive 
mappings such that
\[F:=\bigcap_{n\in\N} Fix(T_n) \neq \emptyset.\]

The generalized Viscosity Approximation Method (genVAM) is defined as follows: 
\begin{equation}\label{def-genVAM}
genVAM \qquad x_0\in C, \qquad x_{n+1}=\alpha_n f(x_n) +(1-\alpha_n)T_nx_n,
\end{equation}
where $(\alpha_n)_{n\in\N}$ is a sequence in $[0,1]$.

Let $z\in F$ and take $K_z \in\N^*$ such that 
\begin{align}\label{def-Kz-main}
K_z \geq \max\left\{d(x_0,z),\frac{d(f(z),z)}{1-\alpha}\right\}.
\end{align}

\begin{lemma}\label{xn-bound-as-reg-hyp}
For all $m,n\in\N$, 
\begin{enumerate}
\item\label{dx-Tnx-bound-hyp}  $d(y,T_ny) \leq 2d(y,z)$ for all $y\in C$;
\item\label{xn-bound-hyp} $d(x_n,z) \le K_z$;
\item\label{xn-bound-hyp-Tm-f} $d(T_mx_n,z), d(f(x_n),z)\le K_z$;
\item\label{xn-Tnm-as-reg-hyp} $d(x_{n+1},x_n), d(T_mx_n,x_n), d(T_mx_n,f(x_n)) \leq 2K_z$.
\end{enumerate}
\end{lemma}
\begin{proof}
\begin{enumerate}
\item We have that $d(y,T_ny) \leq d(y,z) + d(z, T_ny) = d(y,z) + d(T_nz, T_ny) \leq 2 d(y,z)$, as $z\in F$ and $T_n$ is nonexpansive.
\item The proof is by induction on $n$. The case $n=0$ is obvious, by the definition of $K_z$. \\
$n \Rightarrow n+1$:
\begin{align*}
d(x_{n+1},z)&=d(\alpha_n f(x_n) +(1-\alpha_n)T_nx_n,z)   \stackrel{(W1)}{\le} \alpha_nd(f(x_n),z)+(1-\alpha_n)d(T_nx_n,z)\\
&\le \alpha_n (d(f(x_n),f(z))+d(f(z),z))+(1-\alpha_n)d(T_nx_n,z)\\
&\le \alpha_n (\alpha d(x_n,z) + d(f(z),z))+(1-\alpha_n)d(T_nx_n,z)  \quad  \text{as~}f \text{~is an~} \alpha\text{-contraction}\\
&= \alpha_n (\alpha d(x_n,z) + d(f(z),z))+(1-\alpha_n)d(T_nx_n,T_n z)   \quad  \text{as~} z\in F \\
&\le (\alpha_n\alpha+ 1-\alpha_n)d(x_n,z)+\alpha_nd(f(z),z) \quad \text{as~}  T_n \text{~is nonexpansive} \\
&= (1-(1-\alpha)\alpha_n)d(x_n,z)+(1-\alpha)\alpha_n\frac{d(f(z),z)}{1-\alpha}\\
&\le (1-(1-\alpha)\alpha_n)K_z + (1-\alpha)\alpha_n K_z \quad 
\text{by the induction hypothesis and \eqref{def-Kz-main}}\\
&=K_z.
\end{align*}
\item We have that $d(T_mx_n,z) = d(T_mx_n,T_mz) \le d(x_n,z) \stackrel{\eqref{xn-bound-hyp}}{\le} K_z$. Furthermore, 
\begin{align*}
d(f(x_n),z) & \le d(f(x_n),f(z))+d(f(z),z) \le \alpha d(x_n,z)+d(f(z),z) \stackrel{\eqref{xn-bound-hyp}}{\le} \alpha K_z+d(f(z),z) \\
& \stackrel{\eqref{def-Kz-main}}{\le} \alpha K_z +(1-\alpha)K_z = K_z.
\end{align*}
\item Use \eqref{xn-bound-hyp} and  \eqref{xn-bound-hyp-Tm-f} to get that 
$d(x_{n+1},x_n)  \le d(x_n,z)+d(x_{n+1},z) \le 2K_z$, 
$d(T_mx_n,x_n) \le d(T_mx_n,z) + d(z,x_n) \le 2K_z$ and 
$d(T_mx_n,f(x_n)) \le d(T_mx_n,z) + d(z,f(x_n)) \le 2K_z$.
\end{enumerate}
\end{proof}

\begin{lemma}
For all $n\in\N$, 
\begin{equation}\label{dxnTnxn-leq}
d(x_n, T_nx_n) \leq d(x_n, x_{n+1}) + 2K_z\alpha_n.
\end{equation}
\end{lemma}
\begin{proof}
Let $n\in \N$. We have that 
\begin{align*}
d(x_n, T_nx_n) & \leq d(x_n, x_{n+1}) + d(x_{n+1}, T_nx_n)
= d(x_n, x_{n+1}) + d(\alpha_n f(x_n) +(1-\alpha_n)T_nx_n, T_nx_n) \\
& \stackrel{\eqref{W10-W-hyp}}{=} d(x_n, x_{n+1}) + \alpha_n d(f(x_n),T_nx_n)  \leq d(x_n, x_{n+1}) + 2K_z\alpha_n \quad \text{by Lemma \ref{xn-bound-as-reg-hyp}.\eqref{xn-Tnm-as-reg-hyp}}.
\end{align*}
\end{proof}

\subsection{Resolvent-like conditions for $(T_n)$}\label{subsection-resolvent-like}

For every sequence $(\lambda_n)$ of positive reals and $m,n\in\N$, we say that 
the family $(T_n)$ satisfies $Res((\lambda_n),n,m)$ if the following holds: 
\begin{equation}\label{def-Res-lambda-n-m}
d(T_n y, T_m y) \leq \left| 1-\frac{\lambda_m}{\lambda_n} \right|d(y, T_n y) \quad \text{for all $y\in C$}.
\end{equation}

The second author, Nicolae and Sipo\c s introduced condition \eqref{def-Res-lambda-n-m} as condition (C1) in \cite{LeuNicSip18}, 
where they study abstract versions of the proximal point algorithm. 
Examples of families of mappings $(T_n)$ that satisfy $Res((\lambda_n),n,m)$ for all $m,n\in\N$ are:
\begin{enumerate}
\item families $(T_n)$ that are jointly $(P_2)$ or jointly firmly nonexpansive with respect to $(\lambda_n)$ in the setting of CAT(0) spaces. These 
families of mappings are defined in \cite{LeuNicSip18} and studied further in \cite{Sip22,Sip22a}.
\item the family $(J_{\lambda_nf})$  of proximal mappings of a proper convex lower semicontinuous function $f : X \to (-\infty, +\infty]$, 
where $X$ is a complete CAT(0) space (by \cite[Proposition~3.17]{LeuNicSip18}).
\item the family $(R_{T,\lambda_n})$ of resolvents of a nonexpansive mapping $T:X\to X$ in a complete CAT(0) space $X$ (by \cite[Proposition~3.19]{LeuNicSip18}).
\item the family $(J_{\lambda_nA})$  of resolvents of a maximally monotone operator $A:X \to 2^X$ in a Hilbert space $X$ (by \cite[Proposition~3.21]{LeuNicSip18}).
\item the family $(J_{\lambda_nA})$  of resolvents of an $m$-accretive operator $A:X \to 2^X$ in a normed space $X$ (by \cite[Lemma 2.1]{FirLeu25}).
\end{enumerate}

As pointed out in \cite[Lemma 3.1]{LeuNicSip18}, if $(T_n)$ satisfies $Res((\lambda_n),n,m)$ for all $m,n\in\N$, then $Fix(T_n)=F$ for all $n\in \N$.

\begin{proposition}
Let $n\in\N$ and $\beta_n= 1-(1-\alpha)\alpha_{n+1}$. 
\begin{enumerate}
\item Assume that $(T_n)$ satisfies $Res((\lambda_n),n,n+1)$. Then 
\begin{align}
d(x_{n+2},x_{n+1}) \le & \beta_nd(x_{n+1},x_n) +2K_z\left(|\alpha_{n+1}-\alpha_n|+ (1-\alpha_{n+1})\left|1-\frac{\lambda_{n+1}}{\lambda_n} \right|\right).
\label{main-ineq-apply-Xu-hyp}
\end{align}
\item Assume that $(T_n)$ satisfies $Res((\lambda_n),n+1,n)$. Then 
\begin{align}
d(x_{n+2},x_{n+1}) \le & \beta_nd(x_{n+1},x_n) +2K_z\left(|\alpha_{n+1}-\alpha_n|+ (1-\alpha_{n+1})\left|1-\frac{\lambda_n}{\lambda_{n+1}} \right|\right).
\label{main-ineq-apply-Xu-hyp-star}
\end{align}
\end{enumerate}
\end{proposition}
\begin{proof}
We have that 
\begin{align*}
d(x_{n+2}, x_{n+1}) &= d(\alpha_{n+1}f(x_{n+1}) + (1-\alpha_{n+1})T_{n+1}x_{n+1}, \alpha_nf(x_n) + (1-\alpha_n)T_nx_n)\\
& \stackrel{\eqref{W42-xyzw-ineq1-2}}{\le}  \alpha_{n+1}d(f(x_{n+1}),f(x_n))+(1-\alpha_{n+1})d(T_{n+1}x_{n+1},T_nx_n)\\
& \quad  + |\alpha_{n+1}-\alpha_n|d(f(x_n),T_nx_n)\\
& \le \alpha\alpha_{n+1}d(x_{n+1},x_n)+(1-\alpha_{n+1})d(T_{n+1}x_{n+1},T_nx_n)\\
& \quad  + |\alpha_{n+1}-\alpha_n|d(f(x_n),T_nx_n) \quad \text{as $f$ is an $\alpha$-contraction}\\
& \le \alpha\alpha_{n+1}d(x_{n+1},x_n)+(1-\alpha_{n+1})\left(d(T_{n+1}x_{n+1},T_{n+1}x_n)+d(T_{n+1}x_n,T_nx_n)\right)\\
& \quad  + |\alpha_{n+1}-\alpha_n|d(f(x_n),T_nx_n)\\
& \le \alpha\alpha_{n+1}d(x_{n+1},x_n)+(1-\alpha_{n+1})d(x_{n+1},x_n)+(1-\alpha_{n+1})d(T_{n+1}x_n,T_nx_n)\\
& \quad  + |\alpha_{n+1}-\alpha_n|d(f(x_n),T_nx_n)\quad \text{as $T_{n+1}$ is nonexpansive}\\
& \le \beta_nd(x_{n+1},x_n) + (1-\alpha_{n+1})d(T_{n+1}x_n,T_nx_n) + |\alpha_{n+1}-\alpha_n|d(f(x_n),T_nx_n).
\end{align*}
Apply now Lemma \ref{xn-bound-as-reg-hyp}.\eqref{xn-Tnm-as-reg-hyp} to get that 
\begin{equation}\label{dxns-xn1-int}
d(x_{n+2}, x_{n+1})  \le \beta_nd(x_{n+1},x_n) + (1-\alpha_{n+1})d(T_{n+1}x_n,T_nx_n) + 2K_z|\alpha_{n+1}-\alpha_n|.
\end{equation}

\begin{enumerate}
\item By $Res((\lambda_n),n,n+1)$ and Lemma \ref{xn-bound-as-reg-hyp}.\eqref{xn-Tnm-as-reg-hyp} we obtain that 
\begin{align*}
d(T_nx_n,T_{n+1}x_n) \le \left|1-\frac{\lambda_{n+1}}{\lambda_n}\right|d(x_n,T_nx_n) \leq  2K_z\left|1-\frac{\lambda_{n+1}}{\lambda_n} \right|.
\end{align*}
Then \eqref{main-ineq-apply-Xu-hyp} follows immediately from  \eqref{dxns-xn1-int}.
\item We get \eqref{main-ineq-apply-Xu-hyp-star} in a similar way. 
\end{enumerate}
\end{proof}

\begin{proposition}
Let $m,n\in\N$ and assume that $(T_n)$ satisfies $Res((\lambda_n),n,m)$. Then
\begin{align}
d(T_mx_n,x_n) &\le \left(\left| 1-\frac{\lambda_m}{\lambda_n}\right|+1\right)d(T_n x_n,x_n), \label{res-lambda-m-n-main} \\
d(T_mx_n,x_n) &\le \frac{\lambda_m}{\lambda_n}d(T_n x_n,x_n) \quad \text{if } \lambda_m \ge \lambda_n, \label{res-lambda-m-geq-n} \\
d(T_mx_n,x_n) &\le \left(2-\frac{\lambda_m}{\lambda_n}\right)d(T_n x_n,x_n) \quad \text{if } \lambda_m < \lambda_n. \label{res-lambda-m-less-n}
\end{align}
\end{proposition}
\begin{proof}
\eqref{res-lambda-m-n-main} holds, as 
\begin{align*}
d(T_mx_n,x_n) &\le d(T_mx_n,T_nx_n)+d(T_nx_n,x_n) \stackrel{\eqref{def-Res-lambda-n-m}}{\le}  \left| 1-\frac{\lambda_m}{\lambda_n} \right|d(T_n x_n,x_n) + d(T_nx_n,x_n) \\
&=\left(\left| 1-\frac{\lambda_m}{\lambda_n}\right|+1\right)d(T_n x_n,x_n). 
\end{align*}
\eqref{res-lambda-m-geq-n} and \eqref{res-lambda-m-less-n} follow immediately.
\end{proof}

\subsection{Quantitative hypotheses on the parameter sequences}

The following quantitative hypotheses on the parameter 
sequences $(\alpha_n) \subseteq [0,1]$, $(\lambda_n) \subseteq (0,\infty)$ will be used to compute, for the genVAM iteration, 
rates of asymptotic regularity in Section \ref{section-rates-as-reg} and rates of metastability in Section \ref{rates-meta-thms}. 

\begin{align*}
\mathrm{(H1}\alpha_n\mathrm{)} & \quad \sum\limits_{n=0}^{\infty} \alpha_n =\infty \text{~with rate of divergence~} \sigma_1;\\[1mm]
\mathrm{(H2}\alpha_n\mathrm{)} & \quad \sum\limits_{n=0}^{\infty} |\alpha_n-\alpha_{n+1}| <\infty \text{~with Cauchy modulus~} \sigma_2;\\[1mm]
\mathrm{(H3}\alpha_n\mathrm{)} & \quad \lim\limits_{n\to\infty}\alpha_n=0 \text{~with rate of convergence~} \sigma_3;\\[1mm]
\mathrm{(H1}\lambda_n\mathrm{)} & \quad\sum\limits_{n=0}^{\infty} \left|1-\frac{\lambda_{n+1}}{\lambda_n}\right|<\infty \text{~with Cauchy modulus~} \theta_1;\\[1mm]
\mathrm{(H1}^*\lambda_n\mathrm{)} & \quad\sum\limits_{n=0}^{\infty} \left|1-\frac{\lambda_n}{\lambda_{n+1}}\right|<\infty \text{~with Cauchy modulus~} \theta_1^*;\\[1mm]
\mathrm{(H2}\lambda_n\mathrm{)} & \quad \Lambda\in\N^* \text{ and } N_{\Lambda}\in\N \text{~are such that~} \lambda_n \geq 
\frac1\Lambda \text{~for all~} n\geq N_{\Lambda}; \\[1mm]
\mathrm{(H3}\lambda_n\mathrm{)} & \quad\sum\limits_{n=0}^{\infty} |\lambda_n-\lambda_{n+1}| <\infty \text{~with Cauchy modulus~} \theta_2; \\[1mm]
\mathrm{(H4}\lambda_n \mathrm{)} & \quad\lim\limits_{n\to\infty} \lambda_n=\lambda >0 \text{~with rate of convergence~} \theta_4.
\end{align*}
Furthermore, for $(\alpha_n) \subseteq (0,1]$, we shall also consider 
\begin{align*}
\mathrm{(H4}\alpha_n\mathrm{)} & \quad \lim\limits_{n\to\infty} \frac{|\alpha_{n+1}-\alpha_n|}{\alpha_n^2} =0 \text{~with rate of convergence~} \sigma_4;\\[1mm]
\end{align*}

\begin{lemma}\label{H1alphan-prop}
Assume that (H1$\alpha_n$) holds. Then
\begin{enumerate}
\item\label{H1alphan-prop-sigma1} $\sigma_1(n)\geq n-1$ for all $n\in \N$.
\item\label{H1alphan-prop-1} $\sum\limits_{n=0}^\infty (1-\alpha)\alpha_n$ diverges with rate 
$\theta(n)=\sigma_1\left(\left\lceil\frac{n}{1-\alpha}\right\rceil\right)$.
\item\label{H1alphan-prop-2} $\sum\limits_{n=0}^\infty (1-\alpha)\alpha_{n+1}$ diverges with rate 
$\theta^*(n)=\sigma_1^+\left(\left\lceil\frac{n}{1-\alpha}\right\rceil+1\right)\remin 1$.
\end{enumerate}
\end{lemma}
\begin{proof}
\begin{enumerate}
\item By Lemma \ref{rate-divergence-an-prop}.\eqref{rate-divergence-an-01}. 
\item  Apply Lemma \ref{rate-divergence-an-prop}.\eqref{rate-divergence-cbn-np1-cbn} with 
 $a_n:=\alpha_n$,  $\theta:=\sigma_1$ and $c:=1-\alpha$.
\item By Lemma \ref{rate-divergence-an-prop}.\eqref{rate-divergence-an-N} with $N:=1$ and the fact that $(\alpha_n) \subseteq [0,1]$, 
we get that $\sum\limits_{n=0}^\infty\alpha_{n+1}$ diverges with rate 
$\sigma_1^*(n)=\sigma_1^+(n+1)\remin 1$.  Apply now Lemma \ref{rate-divergence-an-prop}.\eqref{rate-divergence-cbn-np1-cbn} with 
 $a_n:=\alpha_{n+1}$, $\theta:=\sigma_1^*$ and $c:=1-\alpha$ to get (iii).
\end{enumerate}
\end{proof}

\begin{lemma}\label{H4lambdan-Caucy-mod}
Assume that (H4$\lambda_n$) holds. Then $\theta_4^*(k)=\theta_4(2k+1)$ is a Cauchy modulus of $(\lambda_n)$. 
\end{lemma}
\begin{proof}
By Lemma \ref{rate-conv-implies-Cauchy-modulus}.
\end{proof}

The following property is proved in \cite[Lemma 4.4]{FirLeu25}.

\begin{lemma}\label{H2lambda+H3lambda-implies-H1lambda-H1*lambda}
Assume that (H2$\lambda_n$) and (H3$\lambda_n$) hold. Then (H1$\lambda_n$), (H1$^*\lambda_n$) hold with 
\[
\theta_1(k)=\theta_1^*(k)=\max\{N_\Lambda,\theta_2(\Lambda (k+1)-1)\}.
\]
\end{lemma}

\subsection{An auxiliary sequence $(y_n)$}\label{section-auxiliary-yn}

Assume, moreover, that $X$ is complete, $C$ is closed, $(\alpha_n) \subseteq (0,1]$ and $(T_n)$ satisfies $Res((\lambda_n),n,m)$ for all $m,n\in \N$. 
Suppose that  (H4$\lambda_n$) holds, that is, $\lim\limits_{n\to\infty} \lambda_n=\lambda >0$ with rate of convergence $\theta_4$,
and let $l\in\N^*$ be such that 
\begin{equation}\label{def-l-lambda}
\lambda > \frac{1}{l+1}. 
 \end{equation}
Then
 \begin{equation}\label{llambdan-ge-lambda-l}
\lambda_n \ge \lambda - \frac{1}{l+1}>0 \quad \text{for all }n\geq \theta_4(l). 
 \end{equation}
 
\begin{lemma}
For every $x\in C$, the sequence $(T_nx)_{n\in \N}$ is Cauchy with modulus $\gamma$ given by 
\begin{align*}
\gamma(k)  = \begin{cases} 
0 & \text{ if } x\in F,\\
\max\{\theta_4(l), \theta_4(2\lceil L_{x,z}(k+1)\rceil-1)\} & \text{ if } x\notin F,
\end{cases}
\end{align*}
where $L_{x,z}=\frac{2(l+1)d(x,z)}{\lambda(l+1)-1}$. 
\end{lemma}
\begin{proof}
The case $x\in F$ is obvious. Assume that $x\notin F$, so $d(x,z)>0$. Let $k,n\in \N$ with $n\geq\gamma(k)$. 
Then for all $p\in \N$, 
\begin{align*}
d(T_nx,T_{n+p}x) & \leq \left| 1-\frac{\lambda_{n+p}}{\lambda_n} \right|d(x, T_nx)
=\frac{|\lambda_n-\lambda_{n+p}|}{\lambda_n} d(x, T_nx)\quad \text{by \eqref{def-Res-lambda-n-m}} \\
& \leq  \frac{|\lambda_n-\lambda_{n+p}|}{\lambda_n}2d(x,z)\quad \text{by  Lemma \ref{xn-bound-as-reg-hyp}.\eqref{dx-Tnx-bound-hyp}}\\
& \leq  \frac{|\lambda_n-\lambda_{n+p}|}{\lambda - \frac{1}{l+1}}2d(x,z) \quad \text{by  \eqref{llambdan-ge-lambda-l}, as } n\geq \theta_4(l)\\
& = L_{x,z} |\lambda_n-\lambda_{n+p}|.
\end{align*}
Apply Lemma \ref{H4lambdan-Caucy-mod} to obtain that $\theta_4^*(k)=\theta_4(2k+1)$ is a Cauchy modulus of $(\lambda_n)$. 
Since $n\geq \theta_4(2\lceil L_{x,z}(k+1)\rceil-1)=\theta_4^*(\lceil L_{x,z}(k+1)\rceil-1)$, we  get that 
$|\lambda_n-\lambda_{n+p}| \leq \frac1{\lceil L_{x,z}(k+1)\rceil}$,
hence $L_{x,z}|\lambda_n-\lambda_{n+p}| \leq \frac1{k+1}$.
\end{proof}

As $C$ is complete, we get that $(T_nx)$ is convergent. Let us define
\begin{equation}
\tilde{T}: C \to C, \qquad \tilde{T}x=\lim\limits_{n\to\infty} T_nx.
\end{equation}
One can easily see that $\tilde{T}$ is nonexpansive and that  $F \subseteq Fix(\tilde{T})$.

\begin{proposition} 
For all $x\in C$ and all $k\in\N$,
\begin{align}
d(T_kx,\tilde{T}x) & \le \frac{|\lambda_k-\lambda|}{\lambda}d(x,\tilde{T}x) 
\le (l+1)|\lambda_k-\lambda|d(x,\tilde{T}x),\label{Tk-tT-ineq-1}  \\
d(T_kx,\tilde{T}x)  & \le \frac{|\lambda_k-\lambda|}{\lambda_k}d(x, T_kx). \label{Tk-tT-ineq-2}
\end{align}
\end{proposition}
\begin{proof}
Let $x\in C$ and  $k\in\N$. Apply \eqref{def-Res-lambda-n-m} to get that for all $n\in \N$,
\begin{align*} 
d(T_k x, T_n x) \leq \left| 1-\frac{\lambda_k}{\lambda_n} \right|d(x, T_n x) \quad \text{and}  \quad d(T_k x, T_n x) \leq \left| 1-\frac{\lambda_n}{\lambda_k} \right|d(x, T_k x). 
\end{align*}
It follows that 
\begin{align*} 
d(T_kx,\tilde{T}x) & =d\left(T_kx,\lim\limits_{n\to\infty} T_nx \right)= \lim\limits_{n\to\infty} d(T_kx,T_nx) 
\le \lim\limits_{n\to\infty} \left| 1-\frac{\lambda_k}{\lambda_n} \right|d(x, T_n x)\\ 
& =\left| 1-\frac{\lambda_k}{\lambda} \right|d(x, \tilde{T}x)
 = \frac{|\lambda_k-\lambda|}{\lambda}d(x,\tilde{T}x) \stackrel{\eqref{def-l-lambda}}{\le} (l+1)|\lambda_k-\lambda|d(x,\tilde{T}x) 
\end{align*} 
and 
\begin{align*} 
d(T_kx,\tilde{T}x) & = d\left(T_kx,\lim\limits_{n\to\infty} T_nx \right) =\lim\limits_{n\to\infty} d(T_kx,T_nx) 
\le \lim\limits_{n\to\infty}\left| 1-\frac{\lambda_n}{\lambda_k} \right|d(x, T_k x) \\
&=\left| 1-\frac{\lambda}{\lambda_k} \right|d(x, T_kx)=\frac{|\lambda_k-\lambda|}{\lambda_k}d(x, T_kx).
\end{align*} 
\end{proof}

\begin{proposition} \label{F=FixtT}
 $F = Fix(\tilde{T})$.
\end{proposition}
\begin{proof}
Obviously, by the definition of $\tilde{T}$ and \eqref{Tk-tT-ineq-1}. 
\end{proof}

Define, for every $n\in \N$, 
\begin{equation}\label{def-Sn}
S_n: C \to C, \quad S_n(x)=\alpha_nf(x) +(1-\alpha_n)\tilde{T}x. 
\end{equation}
One can easily see  that $S_n$ is a $\delta_n$-contraction, where $\delta_n=\alpha_n \alpha + 1-\alpha_n \in [0,1)$, hence we can apply 
the Banach contraction principle to get that  $S_n$ has a unique fixed point $y_n$. 
Thus, for all $n\in \N$, 
\begin{equation}\label{eq-def-yn}
y_n=\alpha_nf(y_n) +(1-\alpha_n)\tilde{T}y_n.
\end{equation}
Furthermore, by \eqref{W10-W-hyp}, 
\begin{equation}
d(y_n,f(y_n))= (1-\alpha_n)d(f(y_n),\tilde{T}y_n) \quad \text{ and } \quad d(y_n,\tilde{T}y_n)= \alpha_nd(f(y_n),\tilde{T}y_n). \label{dyn-fyn-tildeTyn}
\end{equation}

\begin{lemma}\label{yn-bounds-hyp}
For all $n\in \N$,
\begin{align} 
d(y_n,z) & \le \frac{1}{1-\alpha}d(f(z),z)\le K_z, \label{yn-bounds-hyp-1}\\
d(\tilde{T}y_n,f(y_n)),d(y_n,f(y_n)),d(y_n,\tilde{T}y_n) & \le \frac{2}{1-\alpha}d(f(z),z) \le 2K_z. \label{yn-bounds-hyp-2-more}
\end{align} 
\end{lemma}
\begin{proof}
We have that  
\begin{align*} 
d(y_n,z) & \stackrel{\eqref{eq-def-yn}}{=} d(\alpha_nf(y_n) +(1-\alpha_n)\tilde{T}y_n,z) 
\stackrel{(W1)}{\le}  \alpha_nd(f(y_n),z)+(1-\alpha_n)d(\tilde{T}y_n,z) \\
& \le  \alpha_nd(f(y_n),z)+(1-\alpha_n)d(y_n,z) \quad \text{as } \tilde{T}z=z \text{ and }\tilde{T} \text{ is nonexpansive}\\
& \le  \alpha_nd(f(y_n),f(z))+\alpha_nd(f(z),z)+(1-\alpha_n)d(y_n,z)\\
& \le \alpha_n\alpha d(y_n,z)+\alpha_nd(f(z),z)+(1-\alpha_n)d(y_n,z) \\
& =  (1-(1-\alpha)\alpha_n)d(y_n,z)+\alpha_nd(f(z),z).
\end{align*} 
Thus, $(1-\alpha)\alpha_n d(y_n,z) \le \alpha_nd(f(z),z)$, hence the first inequality in \eqref{yn-bounds-hyp-1} holds. Apply \eqref{def-Kz-main}
to get the second inequality. 

Furthermore, 
\begin{align*} 
d(\tilde{T}y_n,f(y_n)) &\le d(\tilde{T}y_n,z)+d(z,f(z))+d(f(z),f(y_n))
\le  (1+\alpha)d(y_n,z)+d(z,f(z))\\
&\stackrel{\eqref{yn-bounds-hyp-1}}{\le}\frac{1+\alpha}{1-\alpha}d(f(z),z)+d(f(z),z) = \frac{2}{1-\alpha}d(f(z),z).
\end{align*}
By \eqref{dyn-fyn-tildeTyn}, we have that $d(y_n,f(y_n)), d(y_n,\tilde{T}y_n)\leq d(\tilde{T}y_n,f(y_n))$. Thus,
\eqref{yn-bounds-hyp-2-more} holds.
\end{proof}

\begin{lemma}
For all $n\in \N$,
\begin{align} 
d(y_n,y_{n+1}) &\le \frac{|\alpha_{n+1}-\alpha_n|}{(1-\alpha)\alpha_n}d(f(y_{n+1}),\tilde{T}y_{n+1}). \label{desig-yn+1-yn}
\end{align} 
\end{lemma}
\begin{proof}
We get that 
\begin{align*} 
d(y_n,y_{n+1})  & \stackrel{\eqref{eq-def-yn}}{=}  d(\alpha_nf(y_n) +(1-\alpha_n)\tilde{T}y_n,\alpha_{n+1}f(y_{n+1})
+(1-\alpha_{n+1})\tilde{T}y_{n+1}) \\
& \stackrel{\eqref{W42-xyzw-ineq1-2}}{\le}  \alpha_n d(f(y_n),f(y_{n+1})) + (1-\alpha_n)d(\tilde{T}y_n,\tilde{T}y_{n+1}) 
+ |\alpha_{n+1}-\alpha_n|d(f(y_{n+1}),\tilde{T}y_{n+1})\\
&\le \alpha\alpha_n d(y_n,y_{n+1})+(1-\alpha_n)d(y_n,y_{n+1})+|\alpha_{n+1}-\alpha_n|d(f(y_{n+1}),\tilde{T}y_{n+1}) \\
&= (1-(1-\alpha)\alpha_n)d(y_{n+1},y_n)+|\alpha_{n+1}-\alpha_n|d(f(y_{n+1}),\tilde{T}y_{n+1}).
\end{align*}
Thus, $(1-\alpha)\alpha_nd(y_n,y_{n+1}) \le |\alpha_{n+1}-\alpha_n|d(f(y_{n+1}),\tilde{T}y_{n+1})$, 
which gives \eqref{desig-yn+1-yn}.
\end{proof}

\begin{lemma}
For all $n\in \N$,
\begin{align} 
d(x_{n+1},y_n) &\le (1-(1-\alpha)\alpha_n) d(x_n,y_n) + (1-\alpha_n)\alpha_n\frac{|\lambda_n-\lambda|}{\lambda}d(f(y_n),\tilde{T}y_n).
\label{desig-xn+1-yn}
\end{align} 
\end{lemma}
\begin{proof}
We have that 
\begin{align*} 
d(x_{n+1},y_n) & = d(\alpha_nf(x_n)+(1-\alpha_n)T_nx_n, \alpha_nf(y_n)+(1-\alpha_n)\tilde{T}y_n) \quad \text{by \eqref{def-genVAM} and \eqref{eq-def-yn}}\\
&\stackrel{(W4)}{\le} \alpha_nd(f(x_n),f(y_n)) + (1-\alpha_n)d(T_nx_n,\tilde{T}y_n)\\
&\le \alpha_n\alpha d(x_n,y_n)+(1-\alpha_n)d(T_nx_n,T_ny_n)+(1-\alpha_n)d(T_ny_n,\tilde{T}y_n) \\
&\le \alpha_n\alpha d(x_n,y_n)+(1-\alpha_n)d(x_n,y_n)+(1-\alpha_n)d(T_ny_n,\tilde{T}y_n) \\
& \stackrel{\eqref{Tk-tT-ineq-1}}{\le} (1-(1-\alpha)\alpha_n) d(x_n,y_n) + (1-\alpha_n)\frac{|\lambda_n-\lambda|}{\lambda}d(y_n,\tilde{T}y_n)\\ 
& \stackrel{\eqref{dyn-fyn-tildeTyn}}{=}  
 (1-(1-\alpha)\alpha_n) d(x_n,y_n) + (1-\alpha_n)\alpha_n\frac{|\lambda_n-\lambda|}{\lambda}d(f(y_n),\tilde{T}y_n).
\end{align*}
\end{proof}

\subsubsection{A rate of convergence to $0$ of  $d(x_n,y_n)$}

\begin{lemma}\label{bound-dxnzn}
For all $n\in \N$, $d(x_n,y_n) \leq 2K_z$. 
\end{lemma}
\begin{proof}
We have that for all $n\in \N$,
\begin{align*}
d(x_n,y_n) & \le d(x_n,z) + d(z,y_n) \le 2K_z \quad \text{by Lemma \ref{xn-bound-as-reg-hyp}.\eqref{xn-bound-hyp} and \eqref{yn-bounds-hyp-1}}.
\end{align*}
\end{proof}

Define, for all $n\in \N$, 
\[
Q_n = \frac{2K_z}{\lambda(1-\alpha)}|\lambda_n-\lambda| + \frac{2K_z}{(1-\alpha)^2}\frac{|\alpha_n-\alpha_{n+1}|}{\alpha_n^2}.
\]

\begin{lemma}
For all $n\in \N$, 
\begin{align}\label{main-ineq-rec-dxnyn}
d(x_{n+1},y_{n+1}) \le (1-(1-\alpha)\alpha_n)d(x_n,y_n) +(1-\alpha)\alpha_n Q_n.
\end{align} 
\end{lemma}
\begin{proof}
We have that 
\begin{align*}
d(x_{n+1},y_{n+1}) &\le d(x_{n+1},y_n) + d(y_n,y_{n+1})\\
&\le (1-(1-\alpha)\alpha_n)d(x_n,y_n) + (1-\alpha_n)\alpha_n\frac{|\lambda_n-\lambda|}{\lambda}d(f(y_n),\tilde{T}y_n)\\
& \quad +\frac{|\alpha_{n+1}-\alpha_n|}{(1-\alpha)\alpha_n}d(f(y_{n+1}),\tilde{T}y_{n+1}) \quad \text{by \eqref{desig-yn+1-yn} and \eqref{desig-xn+1-yn}} \\
&\le  (1-(1-\alpha)\alpha_n)d(x_n,y_n)+ \alpha_n\frac{2K_z}{\lambda}|\lambda_n-\lambda| + \frac{2K_z}{1-\alpha}\frac{|\alpha_n-\alpha_{n+1}|}{\alpha_n}\\
& \quad \text{by  \eqref{yn-bounds-hyp-2-more} and  the fact that } 1-\alpha_n \leq 1 \\
&= (1-(1-\alpha)\alpha_n)d(x_n,y_n) + (1-\alpha)\alpha_n \frac{2K_z}{\lambda(1-\alpha)}|\lambda_n-\lambda| \\
& \quad + (1-\alpha)\alpha_n \frac{2K_z}{(1-\alpha)^2}\frac{|\alpha_n-\alpha_{n+1}|}{\alpha_n^2} \\
&= (1-(1-\alpha)\alpha_n)d(x_n,y_n)+(1-\alpha)\alpha_n Q_n.
\end{align*}
\end{proof}

\begin{proposition} \label{prop-rate-conv-dxnyn}
Assume that (H1$\alpha_n$) and (H4$\alpha_n$) hold.
Define
\begin{equation} \label{prop-rate-conv-dxnyn-psi}
\psi(k) = \max\left\{\theta_4\left(\left\lceil \frac{4K_z(k+1)}{\lambda(1-\alpha)}\right\rceil-1\right), \, 
\sigma_4\left(\left\lceil \frac{{4K_z(k+1)}}{(1-\alpha)^2}\right\rceil-1\right)\right\}.
\end{equation} 
Then $\lim\limits_{n \to \infty} d(x_n,y_n)=0$ with rate of convergence $\Sigma$ given by
\begin{equation}\label{def-rate-conv-dxnyn}
\Sigma(k) = \sigma_1\left(\left\lceil\frac{\psi(2k+1)+\lceil \ln(4K_z(k+1))\rceil}{1-\alpha}\right\rceil\right) + 1.
\end{equation}
\end{proposition}
\begin{proof}
We show that Lemma \ref{quant-lem-Xu02-bncn}.\eqref{quant-lem-Xu02-bncn-cn-0} can be applied with
\begin{equation*}
s_n=d(x_n,y_n),\quad a_n=(1-\alpha)\alpha_n,\quad b_n=Q_n, \quad L=2K_z.
\end{equation*}
First, apply \eqref{main-ineq-rec-dxnyn} to get that $s_{n+1}\leq (1-a_n)s_n + a_nb_n$ for all $n\in\N$. By 
Lemma \ref{bound-dxnzn}, $L$ is an upper bound on $(d(x_n,y_n))$ and,  
by (H1$\alpha_n$) and Lemma \ref{H1alphan-prop}.\eqref{H1alphan-prop-1}, 
 $\sum\limits_{n=0}^\infty a_n$ diverges with rate 
 $\theta(n)=\sigma_1\left(\left\lceil\frac{n}{1-\alpha}\right\rceil\right)$.

As $\lim\limits_{n \to \infty}|\lambda_n-\lambda|=0$ with rate of convergence $\theta_4$ (by (H4$\lambda_n$)) and 
$\lim\limits_{n \to \infty} \frac{|\alpha_{n+1}-\alpha_n|}{\alpha_n^2} =0$ with rate of convergence $\sigma_4$ (by (H4$\alpha_n$)), we can 
apply Lemma \ref{cauchy-conv-rate-linear-comb}.\eqref{cauchy-conv-rate-linear-comb-conv} to get that
 $\lim\limits_{n \to \infty}Q_n =0$ 
with rate of convergence $\psi$.  It follows that $Q_n \le \frac{1}{k+1}$ for all $k,n\in \N$ with $n \ge \psi(k)$.

As the hypotheses of Lemma \ref{quant-lem-Xu02-bncn}.\eqref{quant-lem-Xu02-bncn-cn-0} are satisfied, we apply it to get the conclusion.
\end{proof}

\subsection{Abstract HPPA }

In the hypothesis of Subsection \ref{section-auxiliary-yn}, consider the following iteration: 
\begin{align}\label{def-xnstar}
x^*_0\in C, \qquad x^*_{n+1}=\alpha_n u +(1-\alpha_n)T_nx^*_n,
\end{align}
where $u\in C$ and $(\alpha_n) \subseteq (0,1]$. The iteration $(x^*_n)$ is an abstract version of 
the well-known Halpern-type Proximal Point Algorithm (HPPA), which was
introduced by Xu \cite{Xu02} and Kamimura and Takahashi \cite{KamTak00}.  HPPA is obtained if we take 
$X$ to be a Hilbert space and $T_n=J_{\lambda_n}^A$, where 
$J_{\lambda_n}^A$ is the resolvent of order $\lambda_n$ of a maximally monotone operator $A:X\to 2^X$. 
As pointed out in Subsection \ref{subsection-resolvent-like}, the family $(J_{\lambda_n}^A)$ 
satisfies $Res((\lambda_n),n,m)$ for all $m,n\in \N$. The asymptotic behaviour of this iteration was studied recently, 
with the help of proof mining methods, by Sipo\c s \cite{Sip22} and Kohlenbach and Pinto \cite{KohPin22} 
(they consider the particular case $x^*_0=u$).

As the mapping $f^*:C\to C, \, f^*(x)=u$ is an $\alpha$-contraction with $\alpha=0$, 
\eqref{def-xnstar} is a particular case of the genVAM iteration, obtained by letting $f=f^*$ in 
\eqref{def-genVAM}. Furthermore, by taking $f=f^*$ in \eqref{eq-def-yn}, we get the auxiliary 
sequence $(y^*_n)$ 
satisfying, for all $n\in\N$, 
\begin{equation}\label{eq-def-yns}
y^*_n=\alpha_nu +(1-\alpha_n)\tilde{T}y^*_n.
\end{equation}

Let $z\in F$ and take $K^*_z \in\N^*$ such that 
\begin{align}\label{def-Kzs-main}
K^*_z  \geq \max\left\{d(x^*_0,z), d(u,z)\right\}.
\end{align}

\begin{lemma}\label{bound-dynsz-dynsu}
For all $n\in\N$, $d(y^*_n,z) \leq 2K^*_z$ and $d(y^*_n,u) \leq  3K^*_z$.
\end{lemma}
\begin{proof}
Let $n\in\N$. 
Apply \eqref{yn-bounds-hyp-1} with $f=f^*$, $\alpha=0$, $x_0=x^*_0$ and $K_z=K^*_z$ to get the first inequality. Furthermore,
$d(y^*_n,u) \leq d(y^*_n,z) + d(u,z) \leq 3K^*_z$.
\end{proof}

As an immediate consequence of  Proposition \ref{prop-rate-conv-dxnyn}, we get a rate of 
convergence to $0$ of the sequence $(d(x^*_n,y^*_n))$.

\begin{proposition}\label{prop-rate-conv-dxnsyns-abstractHPPA}
Assume that (H1$\alpha_n$) and (H4$\alpha_n$) hold. Then 
$\lim\limits_{n \to \infty} d(x^*_n,y^*_n)=0$ with rate of convergence $\Sigma^*$ given by
\begin{equation}\label{def-rate-conv-dxnyns}
\Sigma^*(k) = \sigma_1\left(\psi^*(2k+1)+\lceil \ln(4K^*_z(k+1))\rceil\right) + 1,
\end{equation}
where
\begin{equation}
\psi^*:\N\to \N, \quad \psi^*(k) = 
\max\left\{\theta_4\left(\left\lceil \frac{4K^*_z(k+1)}{\lambda}\right\rceil-1\right), \, \sigma_4\left(4K^*_z(k+1)-1\right)\right\}.
\end{equation} 
\end{proposition}

\section{Quantitative asymptotic regularity results}\label{section-rates-as-reg}

The main quantitative results of this section compute, for the genVAM iteration, uniform rates of asymptotic 
regularity (Theorem \ref{main-theorem-1-quant-as-reg}), 
$(T_n)$-asymptotic regularity (Proposition \ref{prop-quant-(Tn)-as-reg} and 
Theorem \ref{main-theorem-2-quant-as-reg}.\eqref{main-theorem-2-quant-as-reg-Tn}) and $T_m$-asymptotic regularity for every $m\in \N$ 
(Proposition \ref{prop-quant-Tm-as-reg}, Theorem \ref{main-theorem-2-quant-as-reg}.\eqref{main-theorem-2-quant-as-reg-tm}).
These results extend to our setting quantitative asymptotic regularity results obtained by the authors \cite{FirLeu25} 
for the VAM iteration associated to resolvents of $m$-accretive operators in Banach spaces, as well as by the author and Pinto \cite{LeuPin21} 
for the HPPA in Hilbert spaces. 
Furthermore, as the abstract HPPA  is a particular case of the genVAM iteration, we get
rates of ($(T_n)$-, $T_m(m\in\N)$-)asymptotic regularity for this iteration, too.
By forgetting the quantitative features, qualitative asymptotic regularity results for the genVAM iteration are immediately obtained 
(see Propositions \ref{main-prop-1-as-reg}, \ref{main-prop-2-as-reg}).

\mbox{}

Throughout this section,  $(X,d,W)$ is a $W$-hyperbolic space,  $C$ is a nonempty convex subset of $X$, 
$f:C\to C$ is an $\alpha$-contraction with $\alpha\in[0,1)$, $(T_n:C\to C)$ is a 
sequence of nonexpansive mappings with $F:=\bigcap_{n\in\N} Fix(T_n) \neq \emptyset$, 
$(\alpha_n) \subseteq [0,1]$, $(x_n)$ is the genVAM iteration defined by \eqref{def-genVAM}, and 
$(\lambda_n) \subseteq (0,\infty)$. Furthermore, we take $z\in F$ and $K_z \in\N^*$ satisfying \eqref{def-Kz-main}.

\begin{theorem}\label{main-theorem-1-quant-as-reg}
Assume that $(T_n)$ satisfies $Res((\lambda_n),n,n+1)$ for all $n\in \N$ and that 
(H1$\alpha_n$), (H2$\alpha_n$), (H1$\lambda_n$) hold. 

Then $(x_n)$ is asymptotically regular with rate $\Phi$ defined by
\begin{equation}\label{quant-as-reg-1-def-Phi}
\Phi(k)= 
\sigma_1^+\left(\left\lceil\frac{\chi(2k+1)+1+\lceil \ln(4K_z(k+1))\rceil}{1-\alpha}\right\rceil+1\right),
\end{equation}
where 
\begin{equation}\label{quant-as-reg-1-def-chi}
\chi:\N\to \N, \quad \chi(k)=\max\{\sigma_2(4K_z(k+1)-1), \theta_1(4K_z(k+1)-1)\}. 
\end{equation}
\end{theorem}
\begin{proof}
Let us verify that the hypotheses of Lemma~\ref{quant-lem-Xu02-bncn}.\eqref{quant-lem-Xu02-bncn-bn-0}  hold with 
\[s_n= d(x_{n+1},x_n), \, \, a_n=(1-\alpha)\alpha_{n+1},   \, \,  
c_n:= 2K_z\left(|\alpha_{n+1}-\alpha_n|+\left|1-\frac{\lambda_{n+1}}{\lambda_n}\right|\right),  \, \,  L:=2K_z.
\]

Use \eqref{main-ineq-apply-Xu-hyp} to get that $s_{n+1}\leq (1-a_n)s_n + c_n$ for all $n\in\N$. 
By Lemma~\ref{xn-bound-as-reg-hyp}.\eqref{xn-Tnm-as-reg-hyp}, $L$ 
is an upper bound on $(s_n)$. Apply (H1$\alpha_n$) and Lemma \ref{H1alphan-prop}.\eqref{H1alphan-prop-2} 
to get that $\sum\limits_{n=0}^\infty a_n$ diverges with rate 
\[\theta(n) = \sigma_1^+\left(\left\lceil\frac{n}{1-\alpha}\right\rceil+1\right)\remin 1.\]
As $c_n= 2K_z|\alpha_{n+1}-\alpha_n|+ 2K_z\left|1-\frac{\lambda_{n+1}}{\lambda_n}\right|$, 
$\sigma_2$ is a Cauchy modulus for $\sum\limits_{n=0}^\infty|\alpha_{n+1}-\alpha_n|$ 
(by (H2$\alpha_n$))
and $\theta_1$ is a Cauchy modulus for 
$\sum\limits_{n=0}^\infty\left|1-\frac{\lambda_{n+1}}{\lambda_n}\right|$ (by (H1$\lambda_n$)), 
we obtain from Lemma \ref{cauchy-conv-rate-linear-comb}.\eqref{series-Cauchy-linear-comb} 
that $\chi$ defined by \eqref{quant-as-reg-1-def-chi} is a Cauchy modulus of $(c_n)$.

We can apply, finally, Lemma~\ref{quant-lem-Xu02-bncn}.\eqref{quant-lem-Xu02-bncn-bn-0},  to get that 
$\lim\limits_{n \to \infty}d(x_{n+1},x_n)=0$ with rate of convergence 
\begin{align*}
\Sigma(k)&= \theta\big(\chi(2k+1)+1+\lceil \ln(4K_z(k+1))\rceil\big)+1 = \left(\Phi(k)\remin 1\right)+1 = \max\left\{\Phi(k), 1\right\}= \Phi(k),
\end{align*}
as $\Phi(k) \geq 1$, by Lemma~\ref{H1alphan-prop}.\eqref{H1alphan-prop-sigma1}.
\end{proof}

\begin{remark}\label{main-theorem-remark-1}
Theorem~\ref{main-theorem-1-quant-as-reg} holds if we replace $Res((\lambda_n),n,n+1)$ with  $Res((\lambda_n),n+1,n)$ and 
(H1$\lambda_n$) with (H1*$\lambda_n$) in the hypothesis  and  $\theta_1$ with $\theta_1^*$ in the rates. 
\end{remark}
\begin{proof}
Use \eqref{main-ineq-apply-Xu-hyp-star} instead of \eqref{main-ineq-apply-Xu-hyp} in 
the proof of Theorem~\ref{main-theorem-1-quant-as-reg}.
\end{proof}

\begin{remark}\label{main-theorem-remark-2}
Theorem~\ref{main-theorem-1-quant-as-reg} holds if,  in the hypothesis, we replace 
(H1$\lambda_n$) with (H2$\lambda_n$) and  (H3$\lambda_n$) and  we take
\begin{equation}\label{def-chi-h2-h3lambda}
\chi(k)= \max\{\sigma_2(4K_z(k+1)-1), N_\Lambda,\theta_2(4K_z\Lambda (k+1)-1)\}.
\end{equation}
\end{remark}
\begin{proof}
Apply Lemma~\ref{H2lambda+H3lambda-implies-H1lambda-H1*lambda} to get that
 (H1$\lambda_n$)  holds with 
$\theta_1(k)=\max\{N_\Lambda,\theta_2(\Lambda (k+1)-1)\}$, so $\chi$  is given by 
\eqref{def-chi-h2-h3lambda}, and follow the proof of Theorem \ref{main-theorem-1-quant-as-reg}.
\end{proof}

\begin{proposition}\label{prop-quant-(Tn)-as-reg}
Let $\Phi$ be a rate of asymptotic regularity of $(x_n)$ and  assume that (H3$\alpha_n$) holds.

Then $(x_n)$ is $(T_n)$-asymptotically regular with rate $\Psi$ defined by
\begin{equation}
\Psi(k)=\max\{\sigma_3(4K_z(k+1)-1),\Phi(2k+1)\}.
\end{equation}
\end{proposition}
\begin{proof}
For all $n \ge \Psi(k)$, 
\begin{align*}
d(T_nx_n , x_n) & \stackrel{\eqref{dxnTnxn-leq}}{\le}   d(x_{n+1},x_n)  + 2K_z\alpha_n \le  \frac{1}{k+1},
\end{align*}
as $n\geq \Phi(2k+1)$, so $d(x_{n+1},x_n)\leq \frac{1}{2(k+1)}$ and  
$n\geq \sigma_3(4K_z(k+1)-1)$, so $2K_z\alpha_n \leq \frac{1}{2(k+1)}$. 
\end{proof}

\begin{proposition}\label{prop-quant-Tm-as-reg}
Let $\Psi$ be a rate of $(T_n)$-asymptotic regularity of $(x_n)$ and $m\in\N$. 
Assume that $(T_n)$ satisfies $Res((\lambda_n),n,m)$ for all $n\in \N$ and that  (H2$\lambda_n$) 
holds.

Then $(x_n)$ is $T_m$-asymptotically regular with rate  $\Psi_m$ defined by
\begin{equation}
\Psi_m(k)=\max\left\{N_\Lambda,\Psi(\Lambda_m\Lambda(k+1)-1),\Psi(2k+1)\right\},
\end{equation}
where $\Lambda_m\in\N^*$ is such that $\Lambda_m\geq \lambda_m$.
\end{proposition}
\begin{proof}
We adapt to our setting the proof of \cite[Theorem 5.4.(ii)]{FirLeu25}. Let $n\ge \Psi_m(k)$. By the definition of $\Psi_m$, we have that 
$\frac1{\lambda_n}\leq \Lambda$, $d(T_nx_n,x_n) \leq \frac1{\Lambda_m\Lambda(k+1)}$ and $d(T_nx_n,x_n) \leq \frac{1}{2(k+1)}$. 

If $\lambda_m \ge \lambda_n$, then $d(T_mx_n ,x_n) \stackrel{\eqref{res-lambda-m-geq-n}}{\le} \frac{\lambda_m}{\lambda_n}d(T_nx_n,x_n) 
\le  \Lambda_m\Lambda d(T_nx_n,x_n) \le  \frac{1}{k+1}$.
If $\lambda_m<\lambda_n$, then $d(T_mx_n,x_n)  \stackrel{\eqref{res-lambda-m-less-n}}{\le}   
\left(2-\frac{\lambda_m}{\lambda_n}\right)d(T_nx_n,x_n) \le 2 d(T_nx_n,x_n) \le  \frac{1}{k+1}$.
\end{proof}

By forgetting the quantitative features of the previous results, we obtain the following qualitative 
asymptotic regularity theorem for genVAM in $W$-hyperbolic spaces.

\begin{proposition}\label{main-prop-1-as-reg}
Assume that 
\begin{center} $\sum\limits_{n=0}^{\infty} \alpha_n = \infty$, \quad 
$\sum\limits_{n=0}^{\infty} |\alpha_n-\alpha_{n+1}| <\infty$,  \quad 
$\lim\limits_{n\to\infty}\alpha_n=0$, \quad  $\liminf\limits_{n\to\infty} \lambda_n>0$
\end{center}
and one of the following holds:
\begin{center}
$\sum\limits_{n=0}^{\infty}|\lambda_n-\lambda_{n+1}| <\infty$, \quad  
$\sum\limits_{n=0}^{\infty} \left|1-\frac{\lambda_{n+1}}{\lambda_n}\right|<\infty$,  \quad  
$\sum\limits_{n=0}^{\infty} \left|1-\frac{\lambda_n}{\lambda_{n+1}}\right|<\infty$.
\end{center}
If $(T_n)$ satisfies $Res((\lambda_n),n,m)$ for all $m,n\in \N$, then 
\begin{center}
$\lim\limits_{n\to\infty} d(x_{n+1},x_n)=\lim\limits_{n\to\infty} d(x_n,T_nx_n)=0$ and 
$\lim\limits_{n\to\infty} d(x_n,T_mx_n)=0$ for all $m\in \N$.
\end{center}
\end{proposition}

\begin{remark}
Assume that $C$ is bounded with diameter $d_C$. Then Theorem \ref{main-theorem-1-quant-as-reg} and 
Propositions \ref{prop-quant-(Tn)-as-reg}, 
\ref{prop-quant-Tm-as-reg} hold with $K_z=\left\lceil \frac{d_C}{1-\alpha}\right\rceil$. 
Thus, the rates of ($(T_n), T_m$-) asymptotic regularity 
of the genVAM iteration do not depend at all on the $W$-hyperbolic space $X$, the family of 
mappings $(T_n)$, the starting point $x_0$ of the 
iteration. They depend via $\alpha$ on the mapping $f$ and  on the subset $C$ only via its diameter $d_C$.
\end{remark}

Inspired by \cite[Theorem 2]{LeuPin21}, we compute additional rates by using the auxiliary sequence 
$(y_n)$ studied in Subsection \ref{section-auxiliary-yn}. 

\begin{theorem}\label{main-theorem-2-quant-as-reg}
Assume that $X$ is complete, $C$ is closed, $(T_n)$ satisfies $Res((\lambda_n),n,m)$ for all $m,n\in \N$, and $(\alpha_n) \subseteq (0,1]$. Suppose that 
(H1$\alpha_n$),  (H3$\alpha_n$), (H4$\alpha_n$) and (H4$\lambda_n$) hold and that $l\in\N^*$ is such that $\lambda > \frac{1}{l+1}$. Then 
\begin{enumerate}
\item\label{main-theorem-2-quant-as-reg-tT} $(x_n)$ is $\tilde{T}$-asymptotically regular with rate 
$\tilde{\Sigma}$ defined by
\begin{equation*}
\tilde{\Sigma}(k)=\max\left\{\sigma_3(4K_z(k+1)-1),\Sigma(4k+3)\right\},
\end{equation*}
where $\Sigma$ is defined by \eqref{def-rate-conv-dxnyn}.
\item\label{main-theorem-2-quant-as-reg-Tn} $(x_n)$ is $(T_n)$-asymptotically regular with rate  
$\Psi$ defined by
\begin{equation*}
\Psi(k)=\max\{\theta_4(l),\tilde{\Sigma}(2k+1)\},
\end{equation*}
\item\label{main-theorem-2-quant-as-reg-tm} For every $m\in \N$, $(x_n)$ is $T_m$-asymptotically regular 
with rate of asymptotic regularity $\Psi_m$ defined by
\begin{equation*}
\Psi_m(k)=\tilde{\Sigma}((1+(l+1)\Lambda^*_m)(k+1)-1),
\end{equation*}
where $\Lambda^*_m \in \N$ is such that $\Lambda^*_m \ge |\lambda-\lambda_m|$.
\end{enumerate}
\end{theorem}
\begin{proof}
Recall that $(y_n)$ is defined by \eqref{eq-def-yn}. Let $k\in \N$ be arbitrary.
\begin{enumerate}
\item For all $n\geq \tilde{\Sigma}(k)$, 
\begin{align*}
d(x_n,\tilde{T}x_n) &\le d(x_n,y_n) + d(y_n, \tilde{T}y_n) + d(\tilde{T}y_n,\tilde{T}x_n)\le 2d(x_n,y_n)+d(y_n, \tilde{T}y_n)\\
&\le \frac{1}{2(k+1)}+ d(y_n, \tilde{T}y_n) \quad \text{by Proposition \ref{prop-rate-conv-dxnyn}, } \text{as }n\geq \Sigma(4k+3)\\
& \stackrel{\eqref{W10-W-hyp}}{=} \frac{1}{2(k+1)} + \alpha_nd(f(y_n),\tilde{T}y_n) 
\stackrel{\eqref{yn-bounds-hyp-2-more}}{\le} \frac{1}{2(k+1)} + 2\alpha_n K_z \\
&\le \frac{1}{k+1} 
\quad \text{ as } n \geq \sigma_3(4K_z(k+1)-1).
\end{align*}
\item For all $n \ge \Psi(k)$, 
\begin{align*}
d(x_n,T_nx_n) &\le d(x_n,\tilde{T}x_n) + d(\tilde{T}x_n,T_nx_n) 
\stackrel{\eqref{Tk-tT-ineq-1}}{\le}  d(x_n,\tilde{T}x_n)+(l+1)|\lambda_n-\lambda|d(x_n,\tilde{T}x_n)  \\
&\le  2d(x_n,\tilde{T}x_n) \quad \text{ as } n \ge \theta_4(l)\\
&\le  \frac{1}{k+1} \quad \text{~since~} n \ge \tilde{\Sigma}(2k+1).
\end{align*}

\item Let $m\in \N$. For all $n \ge \Psi_m(k)$, 
\begin{align*}
d(x_n,T_mx_n) &\le d(x_n,\tilde{T}x_n)+d(\tilde{T}x_n,T_mx_n)
\stackrel{\eqref{Tk-tT-ineq-1}}{\le}  d(x_n,\tilde{T}x_n)+(l+1)|\lambda_m-\lambda|d(x_n,\tilde{T}x_n)\\
&\le d(x_n,\tilde{T}x_n)+(l+1)\Lambda^*_md(x_n,\tilde{T}x_n) \quad \text{by hypothesis}\\
&= (1+(l+1)\Lambda^*_m)d(x_n,\tilde{T}x_n)\le \frac{1}{k+1}  
\quad \text{by \eqref{main-theorem-2-quant-as-reg-tT} and the definition of }\Psi_m(k).
\end{align*}
\end{enumerate}
\end{proof}

As before, we get as an immediate consequence of Theorem \ref{main-theorem-2-quant-as-reg} the 
following qualitative result.

\begin{proposition}\label{main-prop-2-as-reg}
Assume that $X$ is complete, $C$ is closed, $(\alpha_n) \subseteq (0,1]$ and
$(T_n)$ satisfies $Res((\lambda_n),n,m)$ for all $m,n\in \N$. 
Suppose that the following hold: 
\begin{center} 
$\sum\limits_{n=0}^{\infty} \alpha_n = \infty$, $\lim\limits_{n\to\infty}\alpha_n=0$, 
$\lim\limits_{n\to\infty} \frac{|\alpha_{n+1}-\alpha_n|}{\alpha_n^2} =0$ \text{ and }
$\lim\limits_{n\to\infty}\lambda_n = \lambda >0$.
\end{center} 
Then 
\begin{center}
$\lim\limits_{n\to\infty} d(x_n,\tilde{T}x_n)=\lim\limits_{n\to\infty} d(x_n,T_nx_n)=0$ and 
$\lim\limits_{n\to\infty} d(x_n,T_mx_n)=0$ for all $m\in \N$.
\end{center}
\end{proposition}

\begin{corollary}
Let $(x^*_n)$ be the abstract HPPA, defined by \eqref{def-xnstar}. By taking $\alpha:=0$ and $K_z:=K^*_z$ (with $K^*_z\in\N^*$ satisfying \eqref{def-Kzs-main})
in the different rates computed in the previous results, we obtain corresponding rates for $(x^*_n)$.
\end{corollary}

Let $X$ be a normed space, $A:X\to 2^X$ an $m$-accretive operator with a nonempty set of zeros, 
and $f:X\to X$ an $\alpha$-contraction.
The VAM iteration, introduced by Takahashi \cite{Tak07} and studied more recently by Xu et al. \cite{XuAltAlzChe22}, is defined by:
\begin{equation}\label{def-VAM}
 x_0\in X, \qquad x_{n+1}=\alpha_n f(x_n) +(1-\alpha_n)J_{\lambda_n}^Ax_n,
\end{equation}
where  $(\lambda_n) \subseteq (0,\infty)$,  $(\alpha_n) \subseteq [0,1]$ and  $J_{\lambda_n}^A$ is the resolvent of 
order $\lambda_n$ of $A$. By letting $C:=X$ and $T_n:=J_{\lambda_n}^A$ in \eqref{def-genVAM}, we get that the VAM iteration 
is a special case of the genVAM iteration. Furthermore, the family $(T_n)$ satisfies $Res((\lambda_n),n,m)$ for all $m,n\in\N$ 
(see Subsection \ref{subsection-resolvent-like}). 
The authors obtained in \cite[Subsection 5.1]{FirLeu25} quantitative asymptotic regularity results for the VAM iteration. 
These results are immediate consequences of Theorem \ref{main-theorem-1-quant-as-reg} and 
Propositions \ref{prop-quant-(Tn)-as-reg}, \ref{prop-quant-Tm-as-reg}. Furthermore, as corollaries of 
Theorem \ref{main-theorem-2-quant-as-reg} and Proposition \ref{main-prop-2-as-reg}, we obtain 
new quantitative and qualitative results for the VAM iteration. 

\subsection{An example with linear rates}

Take for all $n\in \N$,
\[\alpha_n=\frac{2}{(1-\alpha)(n+J)}, \quad \lambda_n=\frac{n+J}{n+J-1}, 
\quad \text{where }J=2\left\lceil\frac{1}{1-\alpha}\right\rceil.
\]

This choice of parameter sequences was considered by the authors in \cite[Section 2.3]{FirLeu25}
to get linear rates of asymptotic regularity for the VAMe iteration. We show in the sequel that 
we can obtain linear rates for the genVAM iteration by a simple adaptation to our setting of the proofs from \cite[Section 2.3]{FirLeu25}. 

\begin{lemma}\label{lem-as-reg-ex1}
For all $n\in \N$, 
\begin{align}
d(x_{n+1},x_n) & \leq \frac{3JK_z}{(1-\alpha)(n+J)}, \label{linear-ineq-VAM-example-0}\\
d(x_n, T_nx_n) & \leq \frac{(3J+4)K_z}{(1-\alpha)(n+J)}. \label{linear-ineq-VAM-example-1}
\end{align}
\end{lemma}
\begin{proof}
Use \eqref{main-ineq-apply-Xu-hyp}  and follow the proof of \cite[Proposition 5.9]{FirLeu25}(with $e^*=0$) 
to get that for all $n\in\N$,
\begin{align*}
d(x_{n+2},x_{n+1}) & \leq (1-(1-\alpha)\alpha_{n+1})d(x_{n+1},x_n) +  3K_z(\alpha_n-\alpha_{n+1}). 
\end{align*}
Apply Lemma~\ref{lem:sabach-shtern-v2} with $s_n= d(x_{n+1},x_n)$, $a_n=\alpha_n$, $c_n=L=3K_z$,  
$N=2$, $\gamma=1-\alpha$ and $J$ as above to get \eqref{linear-ineq-VAM-example-0}.

Furthermore, for all $n\in \N$,
\begin{align*}
d(x_n, T_nx_n) & \stackrel{\eqref{dxnTnxn-leq}}{\leq} d(x_n, x_{n+1}) + 2K_z\alpha_n 
\stackrel{\eqref{linear-ineq-VAM-example-0}}{\leq} \frac{3JK_z}{(1-\alpha)(n+J)} + \frac{4K_z}{(1-\alpha)(n+J)} 
= \frac{(3J+4)K_z}{(1-\alpha)(n+J)}.
\end{align*}
\end{proof}

\begin{proposition}\label{linear-rates-as-reg-all-ex1}
Define  $\Phi_0, \Psi_0, \Theta_0:\N \to \N$ as follows:
\[\Phi_0(k)=J_0K_z(k+1)-J, \quad \Psi_0(k)=(J_0+2J)K_z(k+1)-J, \quad \Theta_0(k)=(2J_0+4J)K_z(k+1)-J, \]
where  $\displaystyle J_0=\frac{3J^2}{2}=6\left\lceil\frac{1}{1-\alpha}\right\rceil^2$.

Then $\Phi_0$ is a rate of asymptotic regularity of $(x_n)$, $\Psi_0$ is a rate of 
$(T_n)$-asymptotic regularity of $(x_n)$ and 
$\Theta_0$ is a rate of $T_m$-asymptotic regularity of $(x_n)$ for every $m\in\N$.
\end{proposition}
\begin{proof}
We have that for all $n\geq \Phi_0(k)$, 
\begin{align*}
d(x_{n+1},x_n) & \stackrel{\eqref{linear-ineq-VAM-example-0}}{\leq}
\frac{3JK_z}{(1-\alpha)(n+J)} \leq \frac{3JK_z}{(1-\alpha)(\Phi_0(k)+J)}
 = \frac{2}{J(1-\alpha)} \cdot \frac1{k+1}\leq \frac1{k+1}.
\end{align*}
Furthermore, for all $n\geq \Psi_0(k)$, 
\begin{align*}
d(x_n,T_nx_n) & \stackrel{\eqref{linear-ineq-VAM-example-1}}{\leq}\frac{(3J+4)K_z}{(1-\alpha)(n+J)} 
\leq \frac{(3J+4)K_z}{(1-\alpha)(\Psi_0(k)+J)} = \frac{2}{J(1-\alpha)} \cdot \frac1{k+1} \leq \frac1{k+1}.
\end{align*}
Let $m\in\N$. Since (H2$\lambda_n$) holds with $\Lambda=1$, 
$N_\Lambda=0$, and $\lambda_m \leq 2$, we apply Proposition \ref{prop-quant-Tm-as-reg} with 
$\Psi=\Psi_0$ and 
$\Lambda_m=2$ to obtain that $(x_n)$ is $T_m$-asymptotically regular with rate 
$\Theta_0(k)=\Psi_0(2k+1)=(2J_0+4J)K_z(k+1)-J$.
\end{proof}

Let us consider the abstract HPPA $(x^*_n)$, defined by \eqref{def-xnstar}, a particular case of the 
genVAM iteration, obtained by letting $f(x)=u$ and $\alpha=0$. 
Then $J=2$, $\alpha_n=\frac{2}{n+2}$, $\lambda_n=\frac{n+2}{n+1}$, hence we get the example 
considered in \cite[Proposition 3.9]{CheLeu25}. Let 
\[\Phi^*_0(k)=6K_z^*(k+1)-2, \quad \Psi_0(k)=10K_z^*(k+1)-2, \quad \Theta^*_0(k)=20K_z^*(k+1)-2, \]
where $K_z^*\in \N^*$ satisfies \eqref{def-Kzs-main}.
As an immediate consequence of Proposition \ref{linear-rates-as-reg-all-ex1}, we obtain that  $(x^*_n)$ is asymptotically 
regular with rate $\Phi^*_0$, $(T_n)$-asymptotically regular  with rate $\Psi^*_0$, 
and, for all $m\in\N$, $T_m$-asymptotically regular with rate $\Theta^*_0$. Thus, we obtain identical 
rates with the ones from \cite[Proposition 3.10]{CheLeu25}.

\section{Rates of metastability and strong convergence}\label{rates-meta-thms}

In this section, we compute uniform rates of metastability for the abstract HPPA (Theorem \ref{abstract-HPPP-main-thm-quant}) 
and for the genVAM iteration (Theorem \ref{meta-GENVAM-bounded}). As a consequence of these quantitative results, 
we obtain a strong convergence result for the genVAM iteration (Theorem \ref{abstract-genVAM-main-thm-qual}).

In the sequel,  $X$ is a complete CAT(0) space, $\emptyset \neq C \subseteq X$ is closed and convex, $(T_n:C\to C)$ is a 
sequence of nonexpansive mappings with $F:=\bigcap_{n\in\N} Fix(T_n) \neq \emptyset$, 
$f:C\to C$ is an $\alpha$-contraction with $\alpha\in[0,1)$, and $(\alpha_n) \subseteq (0,1]$. 
We assume, moreover, that $(\lambda_n) \subseteq (0,\infty)$ is such that (H4$\lambda_n$) holds and that 
$(T_n)$ satisfies $Res((\lambda_n),n,m)$ for all $m,n\in\N$. 

Furthermore, $(x_n)$ is the genVAM iteration defined by \eqref{def-genVAM}, 
$(x^*_n)$ is the abstract HPPA defined by \eqref{def-xnstar}, 
 $z\in F$, 
$K_z \in\N^*$ satisfies \eqref{def-Kz-main}, and $K^*_z \in\N^*$ satisfies \eqref{def-Kzs-main}. Thus, 
\begin{align*}
K_z \geq \max\left\{d(x_0,z),\frac{d(f(z),z)}{1-\alpha}\right\}, \qquad 
K^*_z  \geq \max\left\{d(x^*_0,z), d(u,z)\right\}.
\end{align*}

The following result, obtained in \cite{KohLeu12a} by adapting to CAT(0) spaces 
Kohlenbach's \cite{Koh11} quantitative version of a well-known theorem of Browder \cite{Bro67}, is an 
essential ingredient for the proof of our first main theorem. This result is a quantitative version of Kirk's \cite{Kir03} 
generalization from Hilbert spaces to CAT(0) spaces of Browder's theorem. 

\begin{proposition} \label{Browder-abstractHPPA-CAT(0)}\cite{KohLeu12a}
Let $(y^*_n)$ be given by \eqref{eq-def-yns} and $M\in \N^*$ satisfy
$M\geq d(y^*_n,u)$ for all $n\in \N$.
\begin{enumerate}
\item If  $(\alpha_n)$ is  nonincreasing, then $(y^*_n)$ is Cauchy with rate of metastability 
$\Omega^*$ given by
\begin{equation}\label{def-Omega-meta-yns-1}
\Omega^*(k,g)=\tilde{g}^{\left(M^2(k+1)^2\right)}(0).
\end{equation}
\item If (H3$\alpha_n$) holds, then $(y^*_n)$ is Cauchy with rate of metastability 
$\Omega^*$ given by
\begin{equation}\label{def-Omega-meta-yns-2}
\Omega^*(k,g)=\sigma_3^+\left(g_h^{\left(4M^2(k+1)^2\right)}(0)\right),
\end{equation}
where $h:\N\to \N$ is such that $\alpha_n\geq \frac1{h(n)+1}$ for all $n\in \N$, and
\[g_{h}:\N\to \N, \quad g_{h}(n)=\max\left\{ h(i)\mid  i \in [0;\tilde{g}(\sigma_3(n))]\right\}.\]
\end{enumerate}
If $C$ is bounded with diameter $d_C$, then one can take $M\geq d_C$. Otherwise, one can take $M=3K^*_z$.
\end{proposition}
\begin{proof}
Remark that $(y^*_n)$ is $(z^u_{t_n})$ from \cite[Section 9]{KohLeu12a} with $t_n=\alpha_n$ and $T=\tilde{T}$.
In \cite[Proposition 9.3]{KohLeu12a}, $C$ is assumed to be bounded and $M\in \N^*$ is an upper bound on 
$d_C$. However, as remarked 
in \cite[Theorem 4.2 ]{Koh11}, instead of assuming that $C$ is bounded, it suffices to assume that
$(y^*_n)$ is bounded and that $M\in \N^*$ satisfies $M\geq d(y^*_n,u)$ for all $n\in \N$. 
By Lemma \ref{bound-dynsz-dynsu}, we can take $M=3K^*_z$.
Apply \cite[Proposition 9.3]{KohLeu12a} and \cite[Remark 9.4.(i)]{KohLeu12a} with 
$\varepsilon=\frac{1}{k+1}$ and $z^u_{t_n}=y^*_n$
to get the result.
\end{proof}

Our first main theorem computes rates of metastability for the abstract HPPA $(x^*_n)$. 
The method of proof extends to this abstract setting the one used by the second author and Pinto in 
\cite{LeuPin21} for the HPPA for maximally monotone operators in Hilbert spaces.

\begin{theorem}\label{abstract-HPPP-main-thm-quant}
Assume that (H1$\alpha_n$) and (H4$\alpha_n$) hold, and in addition that either
$(\alpha_n)$ is nonincreasing or (H3$\alpha_n$) holds.

Then $(x^*_n)$ is Cauchy  with rate of metastability $\Phi^*$ defined by
\begin{equation}\label{def-Phi*-abstract-HPPA}
\Phi^*(k,g)=\max\{\Sigma^*(3k+2),\Omega^*(3k+2,h_{k,g})\},
\end{equation}
where $\Omega^*$ is as in Proposition \ref{Browder-abstractHPPA-CAT(0)},
\begin{align}
\Sigma^*:\N\to \N, & \quad  
\Sigma^*(k) = \sigma_1\left(\psi^*(2k+1)+\lceil \ln(4K^*_z(k+1))\rceil\right) + 1, \label{def-Sigmastar} \\
\psi^*:\N\to \N, & \quad \psi^*(k) = 
\max\left\{\theta_4\left(\left\lceil \frac{4K^*_z(k+1)}{\lambda}\right\rceil-1\right), \, 
\sigma_4\left(4K^*_z(k+1)-1\right)\right\}, \label{def-psistar}
\end{align}
and, for $k\in \N$ and $g:\N \to \N$, 
\begin{equation}\label{def-hkg}
h_{k,g}:\N\to \N, \quad h_{k,g}(m)= \max\{\Sigma^*(3k+2),m\}-m+g(\max\{\Sigma^*(3k+2),m\}).
\end{equation}
\end{theorem}
\begin{proof}
We have that $\Sigma^*$ is a rate of convergence to $0$ of $(d(x^*_n,y^*_n))$ 
(by Proposition \ref{prop-rate-conv-dxnsyns-abstractHPPA}) and that $\Omega^*$ is a rate of 
metastability  of $(y^*_n)$ (by Proposition \ref{Browder-abstractHPPA-CAT(0)}). 
Apply now Proposition \ref{xnyn-lemma-meta} to obtain that $\Phi^*$ is a rate of metastability of 
$(x^*_n)$.
\end{proof}

Sipo\c s \cite[Theorem 4.10]{Sip22} computed for the first time rates of metastability for the abstract HPPA
associated to a family $(T_n)$ that is jointly $(P_2)$ with respect to $(\lambda_n)$, by adapting 
Kohlenbach's quantitative analysis \cite{Koh20a} of the HPPA for accretive operators in Banach spaces.
Theorem \ref{abstract-HPPP-main-thm-quant} extends Sipo\c s's result to families $(T_n)$ satisfying 
$Res((\lambda_n),n,m)$ for all $m,n\in\N$, with a considerably simpler proof and with simpler rates of 
metastability. Furthermore, as an immediate consequence of Theorem \ref{abstract-HPPP-main-thm-quant}, 
we get rates of metastability for the HPPA for maximally monotone operators in Hilbert spaces 
that are similar to, but better than, those obtained in \cite[Corollary 1.(i)]{LeuPin21}.

\begin{corollary}\label{abstract-HPPA-bounded-C}
Assume that $C$ is bounded and  $M\in\N^*$ is an upper bound on the diameter of $C$. 
Let $\Phi^*_M$ be obtained by replacing $K^*_z$ with $M$ in the definition of  $\Phi^*$. Then 
$(x^*_n)$ is  Cauchy  with rate of metastability $\Phi^*_M$.
\end{corollary}

Our second main theorem computes uniform rates of metastability for the genVAM iteration. Its proof uses 
Theorem \ref{abstract-HPPP-main-thm-quant} and methods developed by Kohlenbach and Pinto \cite{KohPin22}.

Following \cite{KohPin22}, for any $x\in C$, let $(w_n(x))$ denote the abstract HPPA $(x^*_n)$  
starting with $x$  and with $u=x$.  Thus,
\begin{equation}\label{def-wnx}
w_0(x) = x, \quad  w_{n+1}(x)=\alpha_nx+(1-\alpha_n)T_nw_n(x). 
\end{equation}

\begin{theorem}\label{meta-GENVAM-bounded}
Assume that $C$ is bounded and let $M\in \N^*$ be an upper bound on the diameter of $C$. 
Suppose that (H1$\alpha_n$)and (H4$\alpha_n$) hold, and in addition that either
$(\alpha_n)$ is nonincreasing or (H3$\alpha_n$) holds.

Then $(x_n)$ is Cauchy with rate of metastability $\Omega$ defined as follows:
\begin{equation}\label{def-Omega-meta-GENVAM-bounded}
\Omega(k,g)=\Psi\left(\frac1{k+1},\tilde{g},0\right),
\end{equation}
where $\Phi^*_M$ is given by Corollary \ref{abstract-HPPA-bounded-C},
\begin{align*}
%\tilde{\Psi}:\N\times \N^\N\times \N \to \N, &\quad \tilde{\Psi}(k,g,n)=\Psi\left(\frac{1}{k+1},g,n\right)\\
\Psi:\R^+\times \N^\N\times \N \to \N, & \quad \Psi(\varepsilon,g,n) =\beta\left(\frac{\varepsilon}{3},\Psi_{\varepsilon,g,n}^{L+1}\right),\\
\Phi^+:\N\times \N^\N\times \N \to \N, & \quad \Phi^+(k,g,n) = n  + \Phi^*_M(k,g_n),\\
\Phi^\dag:\N\times \N^\N\times \N \to \N, & \quad \Phi^\dag(k,g,n) =\max\{\Phi^+(k,g,i) \mid i\in[0;n]\},\\
\tilde{\Phi}: \R^+ \times \N^\N\times \N \to \N, & \quad \tilde{\Phi}(\varepsilon, g, n)=\Phi^{\dag}\left(\left\lceil \frac1\varepsilon\right\rceil -1 ,g,n\right)\\
\sigma^*:\N \to \N, &\quad \sigma^*(n)=\sigma_1^+\left(\left\lceil\frac{n}{1-\alpha}\right\rceil\right), \\
\beta:\R^+\times \N \to \N, &\quad  \beta(\varepsilon,n)=\sigma^*\left(n+\left\lceil \ln\left(\frac{2M}{\varepsilon}\right) \right\rceil\right)+1,
\end{align*}
and for all $\varepsilon >0$, $n\in \N$, $g:\N\to \N$, 
\begin{align*} 
g_n:\N\to\N, & \quad g_n(m)=n+g(n+m), 
\end{align*}
\begin{align*} 
\tilde{\varepsilon}=\frac{\varepsilon(1-\alpha)}{9}, \, \varepsilon_0=\frac{\tilde{\varepsilon}(1-\alpha)}{2}, \,
L=\left\lceil\log_{\alpha}\left(\frac{\tilde{\varepsilon}}{2M}\right)\right\rceil\,
\end{align*}
$\begin{array}{l}
F_{\varepsilon,g}^{i}: \N \to \N  \text{ and } \Psi_{\varepsilon,g,n}^{i}\in \N  
\text{ are defined by induction on $i\in [0;L]$ as follows:}\\[2mm]
F_{\varepsilon,g}^{0}(q)=\max\left\{g\left(\beta\left(\frac{\varepsilon}{3},q\right)\right),\beta\left(\frac{\varepsilon}{3},q\right)\right\}, 
 \quad  F_{\varepsilon,g}^{i+1}(q)=\max\left\{F_{\varepsilon,g}^{0}(q), \tilde{\Phi}\left(\varepsilon_0,F_{\varepsilon,g}^{i},q \right)\right\},\\[1mm]
\Psi_{\varepsilon,g,n}^{0}=n, \quad \Psi_{\varepsilon,g,n}^{i+1}=\tilde{\Phi}\left(\varepsilon_0,F_{\varepsilon,g}^{L-i},\Psi_{\varepsilon,g,n}^{i}\right).
\end{array}$
\end{theorem}
\begin{proof}
As (H1$\alpha_n$) holds,  we have that $\sigma_1^+$ is a  rate of divergence of $\sum\limits_{n=0}^{\infty} \alpha_n$. 
Furthermore,   $\sigma_1^+$ is also monotone in the 
sense of \cite{KohPin22}, that is: for all $m,n\in \N$, $m\leq n $ implies $\sigma_1^+(m) \leq \sigma_1^+(n)$. 

By Corollary \ref{abstract-HPPA-bounded-C}, $\Phi^*_M$ is a rate of metastability of $(w_n(x))$ for all
$x\in C$. 
 Apply Lemma \ref{rate-meta-H-Tn}.\eqref{def-metastability-dag} to get that 
$\Phi^\dag$ is monotone in the third argument and that the following holds for all $x\in C$: 
\[
\forall k,n\in \N\, \forall g:\N \to \N\, \exists N\in [n,\Phi^\dag(k,g,n)]\,\forall i,j \in [N;N+g(N)] 
\left(d(w_i(x), w_j(x)) \leq \frac{1}{k + 1}\right).
\]
Then  $\tilde{\Phi}$ is a monotone (in the sense of \cite[Remark 3.4]{KohPin22}) function that 
satisfies  condition $H[T_n]$ from \cite[pag. 10]{KohPin22}. Thus, we can apply \cite[Theorem 3.10]{KohPin22} with 
\begin{center}
$b=M$, $S_n=T_n$, $A=\sigma_1^+$, $\theta'_b=\tilde{\Phi}$, $\sigma_1=\beta$,  $\varphi=f$, and $\delta(\varepsilon)=1-\alpha$ 
for all $\varepsilon>0$
\end{center}
to get that $\Psi$ as defined in the hypothesis satisfies the following: 
for all $\varepsilon>0 $, $n\in \N$, $g:\N\to \N$ there exist $N\in [n;\Psi(\varepsilon,g,n)]$, $P\in [n;N]$ and $y\in C$ such that 
\begin{equation}\label{meta-GENVAM-bounded-useful}
d(y,w_P(f(y))) \le \varepsilon \, \,  \text{and} \, \, \forall i \in [N;g(N)] \left(d(x_i,y) \le 
\frac{\varepsilon}2\right).
\end{equation}
It follows that the mapping $\Gamma:\N\to \N, \,\, \Gamma(\varepsilon,g)=\Psi(\varepsilon,g,0)$ is a rate of metastability of 
 $(x_n)$ in the sense of \cite{KohPin22}. Apply Remark \ref{relation-rates-meta-KohPin22} to get that 
$\Omega$ is a rate of metastability of $(x_n)$. 
\end{proof}

\begin{fact}
As in \cite[Remark 3.12]{KohPin22}, Theorem \ref{meta-GENVAM-bounded} can be adapted to the case when 
$M\in \N^*$ is an upper bound on the set $\{d(x_0, w_n(f(x_0))) \mid n\in \N\}$, without assuming that $C$ is bounded.
\end{fact}

As a consequence of (the proof of) our quantitative Theorem \ref{meta-GENVAM-bounded} and reasoning as in \cite[Remark 3.6, Remark 3.11]{KohPin22}, 
we get a strong convergence result for the genVAM iteration $(x_n)$.

\begin{theorem}\label{abstract-genVAM-main-thm-qual}
Assume that the following hold:
\begin{center}
$\sum\limits_{n=0}^{\infty} \alpha_n =\infty$, \,\, 
$\lim\limits_{n\to\infty}\alpha_n=0$, \,\, 
$\lim\limits_{n\to\infty} \frac{|\alpha_{n+1}-\alpha_n|}{\alpha_n^2}=0$, and 
$\lim\limits_{n\to\infty}\lambda_n = \lambda >0$.
\end{center}
Then $(x_n)$ converges strongly to the unique fixed point of $Q \circ f$, where 
\[Q:C\to C, \quad Q(x) = \lim\limits_{n\to \infty} w_n(x).\]
\end{theorem}
\begin{proof}
By Theorem \ref{meta-GENVAM-bounded}, $(x_n)$ is Cauchy.  Thus, as $X$ is complete and $C$ 
is closed, $(x_n)$ converges strongly to some $\overline{x}\in C$. Furthermore, for all $x\in C$, 
$(w_n(x))$ converges strongly, hence $Q$ is well-defined.\\[2mm]
\textbf{Claim: } For all $x,y\in C$ and $n\in\N$, $d(w_n(x),w_n(y)) \le d(x,y)$. \\[1mm]
\textbf{Proof of Claim:} The proof is by induction on $n$. The case $n=0$ is obvious, as 
$w_0(x)=x$ and $w_0(y)=y$. For the inductive step, we get that for all $n\in\N$, 
\begin{align*}
d(w_{n+1}(x),w_{n+1}(y)) &=d(\alpha_nx+(1-\alpha_n)T_n w_n(x),\alpha_n y+(1-\alpha_n)T_n w_n(y))\\
&\le  \alpha_nd(x,y) + (1-\alpha_n)d(w_n(x),w_n(y)) \\
& \quad \text{by (W4) and the  nonexpansivity of }T_n \\
&\le  d(x,y)\quad \text{ by the induction hypothesis}. 
\end{align*}
\,\,  \hfill $\blacksquare$

As an immediate consequence of the claim, we get that $Q$ is nonexpansive, hence $Q \circ f$ is 
a contraction. 
We prove in the sequel that $\overline{x}$ is the unique fixed point of $Q \circ f$.

Let $k\in \N$ be arbitrary. As $\lim\limits_{n\to \infty}x_n=\overline{x}$ and 
$Q(f(\overline{x}))=\lim\limits_{n\to \infty}w_n(f(\overline{x}))$, there exists $N_0\in \N$ such that for all 
$ n\ge N_0$, 

\begin{equation}\label{ineq-sc-1}
d(x_n,\overline{x}), d(w_n(f(\overline{x})),Q(f(\overline{x}))) \le \frac{1}{6(k+1)}.
\end{equation}

Apply now \eqref{meta-GENVAM-bounded-useful} with $\varepsilon=\frac{1}{6(k+1)}$ and $n=N_0$ to get that there exist
$N_1,P \ge N_0$ and $y\in C$  such that 
\begin{equation}\label{ineq-sc-2}
d(y,w_P(f(y))), d(x_{N_1},y) \le \frac1{6(k+1)}. 
\end{equation}
Furthermore, as $d(\overline{x},y) \le d(x_{N_1},\overline{x})+d(x_{N_1},y)$, we obtain, using 
\eqref{ineq-sc-1} and \eqref{ineq-sc-2},  that
\begin{equation}\label{ineq-sc-3}
d(\overline{x},y) \le \frac{1}{3(k+1)}.
\end{equation}
We have that 
\begin{align*}
d(\overline{x},Q(f(\overline{x}))) &\le d(\overline{x},w_P(f(\overline{x})))+d(w_P(f(\overline{x})),Q(f(\overline{x}))) 
 \stackrel{\eqref{ineq-sc-1}}{\le} d(\overline{x},w_P(f(\overline{x}))) + \frac{1}{6(k+1)} \\
&\le d(\overline{x},y)+d(y,w_P(f(y)))+ d(w_P(f(y)),w_P(f(\overline{x}))) + \frac{1}{6(k+1)} \\
& \le \frac{4}{6(k+1)} + d(w_P(f(y)),w_P(f(\overline{x})))\quad \text{by \eqref{ineq-sc-2} and 
  \eqref{ineq-sc-3}} \\
& \leq  \frac{4}{6(k+1)} +  d(y, \overline{x}) \quad \text{by the claim and the fact that }f \text{ is a contraction}\\
& \stackrel{\eqref{ineq-sc-3}}{\le} \frac{1}{k+1}.
\end{align*}
As $d(\overline{x},Q(f(\overline{x}))) \leq \frac{1}{k+1}$ for arbitrary $k\in \N$, it follows  that $Q(f(\overline{x}))=\overline{x}$. 
\end{proof}

\subsection{Examples}

We consider first the example given in \cite[Section 5]{LeuPin21}. Let for all $n\in \N$,
\[
\alpha_n=(n+2)^{-\frac{3}{4}}, \quad \lambda_n=1+\frac{(-1)^n}{n+1}.
\]
As pointed out in \cite[Section 5]{LeuPin21}, $(\alpha_n)$ is nonincreasing and (H1$\alpha_n$), 
(H4$\alpha_n$) and (H4$\lambda_n$) hold with 
\[
\sigma_1(k)=(k+1)^4, \quad \sigma_4(k)=(k+1)^4+1, \quad  \lambda=1, \quad  \theta_4(k)=k.
\]

As an application of Theorem \ref{abstract-HPPP-main-thm-quant} we get the following.

\begin{proposition}
$(x^*_n)$ is Cauchy  with rate of metastability $\Phi^*$ defined by
\[
\Phi^*(k,g)=\max\{\Sigma^*(3k+2),\Omega^*(3k+2,h_{k,g})\},
\]
where $\Omega^*$ is defined by \eqref{def-Omega-meta-yns-1} and $h_{k,g}$  by \eqref{def-hkg}, with 
\[
\Sigma^*(k) = \left(2^{12}{K^*_z}^4(k+1)^4 + 2+\lceil \ln(4K^*_z(k+1))\rceil\right)^4 +1.
\]
\end{proposition}

Furthermore, rates of metastability of $(x_n)$ can be computed as an application of Theorem \ref{meta-GENVAM-bounded}.

\mbox{}

In the following, we give another example of parameter sequences $(\alpha_n)$, $(\lambda_n)$  satisfying the hypotheses of 
Theorem \ref{meta-GENVAM-bounded}. Take for all $n\in \N$,
\[
\alpha_n=\tilde{\alpha}+\frac{2}{(1-\alpha)(n+J)}, \qquad \lambda_n=\frac{n+J}{n+J-1}, 
\]
where  $J=2\left\lceil\frac{1}{1-\alpha}\right\rceil+1$ and $\tilde{\alpha}\in\left(0,\frac{1-\alpha}{3-\alpha}\right]$.

As $(\alpha_n)$ is nonincreasing, we get that for all $n\in \N$, $\alpha_n \le \alpha_0 = \tilde{\alpha}+\frac{2}{(1-\alpha)J} \leq  \frac{1-\alpha}{3-\alpha} +
\frac{2}{(1-\alpha)J} \leq 1$. Thus, $(\alpha_n) \subseteq (0,1]$.

\begin{proposition}\label{example2-rate-as-reg}
$(x_n)$ is asymptotically regular with rate $\Phi_0(k)=\left\lceil\frac{3J^2K_z}{2}\right\rceil(k+1)-J$.
\end{proposition}
\begin{proof}
Slightly change the proof of Lemma \ref{lem-as-reg-ex1} to get that for all $n\in \N$, 
\begin{align*}
d(x_{n+2},x_{n+1}) & \leq (1-(1-\alpha)\alpha_{n+1})d(x_{n+1},x_n)+3K_z(\alpha_n-\alpha_{n+1}).
\end{align*}
As $\alpha_{n+1}-\tilde{\alpha} < \alpha_{n+1}$, we get that 
\begin{align*}
d(x_{n+2},x_{n+1}) & \leq (1-(1-\alpha)(\alpha_{n+1}-\tilde{\alpha}))d(x_{n+1},x_n)+3K_z(\alpha_n-\alpha_{n+1})
\end{align*}
for all $n\in \N$. Apply now Lemma \ref{lem:sabach-shtern-v2} with
$s_n= d(x_{n+1},x_n)$, $a_n=\alpha_n-\tilde{\alpha}$, $c_n=L=3K_z$,  
$N=2$, $\gamma=1-\alpha$ and $J$ as above to obtain that 
$d(x_{n+1},x_n) \le \frac{3JK_z}{(1-\alpha)(n+J)}$ for all $n\in \N$.  
The fact that $\Phi_0$ is a rate of asymptotic regularity of $(x_n)$ follows as in 
the proof of Proposition \ref{linear-rates-as-reg-all-ex1}.
\end{proof}

\begin{proposition}
For any $L\in \N^*$, 
\[
\varphi_L: \N \to \N, \qquad \varphi_L(k) = L\left\lceil\frac{3J^2K_z}{2}\right\rceil (k+1)-J
\]
is a rate of $L$-metastability of $(x_n)$.
\end{proposition}
\begin{proof}
Apply Propositions \ref{as-reg-L-meta} and \ref{example2-rate-as-reg}.
\end{proof}

\begin{lemma}\label{example2-quant-hypotheses}
\begin{enumerate}
\item\label{example2-quant-hypotheses-H1} (H1$\alpha_n$) holds with $\sigma_1(n)=\min\left\{4^{\left\lceil \frac{(1-\alpha)n}{2}\right\rceil+J-1}-J,\left\lceil\frac{n}{\tilde{\alpha}}\right\rceil \remin 1\right\}$.
\item\label{example2-quant-hypotheses-H4} (H4$\alpha_n$) holds with
$\sigma_4(k)=\left\lceil\sqrt{\left\lceil\frac{2}{\tilde{\alpha}^2(1-\alpha)}\right\rceil(k+1)}\, \right\rceil \remin J$. 
\item\label{example2-quant-hypotheses-H4l} (H4$\lambda_n$) holds with $\lambda=1$ and 
$\theta_4(k)=(k+2)\remin J$.
\end{enumerate}
\end{lemma}
\begin{proof}
\begin{enumerate}
\item It is well-known that $\sum\limits_{n=0}^{\infty}\frac{1}{n+1}$ diverges with rate 
$\theta(n)=4^n-1$. Apply Lemma \ref{rate-divergence-an-prop}.\eqref{rate-divergence-an-N} 
with $a_n=\frac1{n+1}$ and  $N=J-1$ to obtain that $\sum\limits_{n=0}^{\infty}\frac{1}{n+J}$ 
diverges with rate  
$$\theta^*(n)=\theta(n+J-1)\remin (J-1)=(4^{n+J-1}-1)\remin (J-1)=4^{n+J-1}-J.$$  
Furthermore, by Lemma \ref{rate-divergence-an-prop}.\eqref{rate-divergence-cbn-np1-cbn} with 
$a_n=\frac1{n+J}$ and $c=\frac{2}{1-\alpha}$, we get that $\sum\limits_{n=0}^{\infty}\frac{2}{(1-\alpha)(n+J)}$ 
diverges with rate 
$$\chi(n)=\theta^*\left(\left\lceil \frac{(1-\alpha)n}{2}\right\rceil\right)
=4^{\left\lceil \frac{(1-\alpha)n}{2}\right\rceil+J-1}-J.$$

Obviously, $\sum\limits_{n=0}^{\infty} \tilde{\alpha}$ diverges with rate 
$\gamma(n)=\left\lceil\frac{n}{\tilde{\alpha}}\right\rceil \remin 1$.

Finally, an application of Lemma \ref{rate-divergence-an-prop}.\eqref{rate-divergence-sum-bn} 
gives us the result.

\item Remark first that for all $n\in\N$, 
$$\frac{|\alpha_{n+1}-\alpha_n|}{\alpha_n^2} \leq \frac{2}{\tilde{\alpha}^2(1-\alpha)}\cdot 
\frac{1}{(n+J)(n+1+J)} \leq \frac{2}{\tilde{\alpha}^2(1-\alpha)}\cdot \frac{1}{(n+J)^2}.$$ 

Then for all $n\geq \sigma_4(k)$, 
\begin{align*}
\frac{1}{(n+J)^2}\leq 
\frac{1}{\left\lceil\frac{2}{\tilde{\alpha}^2(1-\alpha)}\right\rceil(k+1)}, \text{ hence } 
\frac{|\alpha_{n+1}-\alpha_n|}{\alpha_n^2}  \leq \frac{1}{k+1}.
\end{align*}

\item For all $n\geq \theta_4(k)$,
\begin{align*}
\left|\frac{n+J}{n+J-1} -1\right| & = \frac{1}{n+J-1} \leq  \frac{1}{((k+2)\remin J)+J-1}
\leq \frac1{k+1}.
\end{align*}
\end{enumerate}
\end{proof}

As $(\alpha_n)$ is nonincreasing and  (H1$\alpha_n$), (H4$\alpha_n$), (H4$\lambda_n$) hold, we can apply 
Theorem \ref{meta-GENVAM-bounded} to compute rates of metastability of $(x_n)$.

\mbox{}

In the sequel, we consider the abstract HPPA $(x^*_n)$, hence $\alpha=0$, and take $\tilde{\alpha}=\frac13$. 
Then  
\[J=3,  \, \,  \,  \alpha_n=\frac13+\frac{2}{n+3}, \, \,  \, \lambda_n=\frac{n+3}{n+2}.\]

We get that for any $L\in \N^*$,  $(x^*_n)$ is $L$-metastable with rate 
\[\varphi^*_L(k) = L\left\lceil\frac{27K^*_z}{2}\right\rceil (k+1)-3. \]
Furthermore, (H1$\alpha_n$) holds with 
$\sigma_1(n)=\min\left\{4^{\left\lceil \frac{n}{2}\right\rceil+2}-3, 3n \remin 1\right\}$, 
(H4$\alpha_n$) holds with $\sigma_4(k)=\left\lceil\sqrt{18(k+1)}\right\rceil-3$, and 
(H4$\lambda_n$) holds with $\lambda=1$ and $\theta_4(k)=k\remin 1$.

\begin{proposition}\label{example2-rates-meta-abstractHPPA}
$(x^*_n)$ is Cauchy with rate of metastability 
\[
\Phi(k,g)=\max\{3\gamma^*(3k+2),\Omega^*(3k+2,h_{k,g})\},
\]
 where $\Omega^*$ is given by \eqref{def-Omega-meta-yns-1}, 
 $\gamma^*, \psi^*:\N\to \N$ are defined as follows:
\begin{align*}
\gamma^*(k) & =  \psi^*(2k+1)+\lceil\ln(4K^*_z(k+1))\rceil, \\
\psi^*(k) & = \max\left\{4K^*_z(k+1)-2, \, \left\lceil 6\sqrt{2K^*_z(k+1)} \,\right\rceil -3 \right\},
\end{align*}
and, for $k\in \N$ and $g:\N \to \N$, 
\[
h_{k,g}:\N\to \N, \quad h_{k,g}(m)= \max\{3\gamma^*(3k+2),m\}-m+g(\max\{3\gamma^*(3k+2),m\}).
\]
\end{proposition}
\begin{proof}
Apply Theorem \ref{abstract-HPPP-main-thm-quant} and remark that 
\begin{align*}
\Sigma^*(k) & = \sigma_1\left(\gamma^*(k)\right) + 1 
= \min\left\{16\cdot 4^{\left\lceil \frac{\gamma^*(k)}{2}\right\rceil}-3, 3\gamma^*(k) - 1\right\} + 1
 = 3\gamma^*(k),
\end{align*}
so $\Sigma^*(3k+2)=3\gamma^*(3k+2)$. 
\end{proof}

 \mbox{}

\noindent \textbf{Acknowledgement} 
We thank the referee for providing comments and suggestion sthat improved the final version of the paper. The first author acknowledges the support of FCT – Fundação
para a Ciência e Tecnologia through a doctoral scholarship with reference number
2022.12585.BD as well as the support of the research center CEMS.UL under the
FCT funding \href{https://doi.org/10.54499/UID/04561/2025}{UIDB/04561/2025}.

\end{document}